\documentclass[11pt,leqno]{amsart}
\usepackage{amsmath,epsfig,graphicx,color, float,amssymb}
\usepackage{subcaption}
\usepackage{relsize}
\usepackage{comment}
\usepackage[colorlinks, linkcolor=blue,citecolor=blue]{hyperref}
\usepackage{pdfpages}
\usepackage{graphicx}
\usepackage{mwe}
\usepackage{algorithm}
\usepackage{algorithmicx}
\usepackage{algpseudocode}
\usepackage{dsfont}

\usepackage[margin=1.3in]{geometry}

\numberwithin{table}{section}
\numberwithin{figure}{section}
\numberwithin{equation}{section}

\definecolor{darkblue}{rgb}{.2, 0.2,.8}
\definecolor{darkgreen}{rgb}{0,0.5,0.3}
\definecolor{darkred}{rgb}{.8, .1,.1}

\newcommand{\dd}{\mathrm{d}}

\newcommand{\psii}{m}
\newcommand{\wh}[1]{\widehat{{#1}}}

\newcommand{\Ta}{\nu_1}
\newcommand{\Tb}{\nu_2}
\newcommand{\Tc}{\nu_3}
\newcommand{\Td}{\nu_4}
\newcommand{\rhoa}{\rho_{-M}}
\newcommand{\rhob}{\rho_0}
\newcommand{\rhoc}{\rho_M}

\newtheorem{lemma}{Lemma}[section]
\newtheorem{theorem}[lemma]{Theorem}

\newtheorem{assumption}{Assumption}

\newtheorem{corollary}[lemma]{Corollary}
\newtheorem{example}{Example}

\newtheorem{remark}{Remark}[section]

\usepackage{bm}

\allowdisplaybreaks

\begin{document}
\title[Space-grid approximations of hybrid stochastic differential equations]{Space-grid approximations of hybrid stochastic differential equations and first passage properties}

\author[H. Albrecher]{Hansjoerg Albrecher}
\address{Faculty of Business and Economics,
University of Lausanne,
Quartier de Chambronne,
1015 Lausanne,
Switzerland}
\email{hansjoerg.albrecher@unil.ch}

\author[O. Peralta]{Oscar Peralta}
\address{Faculty of Business and Economics,
University of Lausanne,
Quartier de Chambronne,
1015 Lausanne,
Switzerland}
\email{oscar.peraltagutierrez@unil.ch}

\begin{abstract}
Hybrid stochastic differential equations are a useful tool to model continuously varying stochastic systems which are modulated by a random environment that may depend on the system state itself. In this paper, we establish the pathwise convergence of the solutions to hybrid stochastic differential equations via space-grid discretizations. While time-grid discretizations are a classical approach for simulation purposes, our space-grid discretization provides a link with multi-regime Markov modulated Brownian motions, leading to computational tractability. We exploit our convergence result to obtain efficient approximations to first passage probabilities and expected occupation times of the solutions hybrid stochastic differential equations, results which are the first of their kind for such a robust framework. We finally illustrate the performanced of the resulting approximations in numerical examples.
\end{abstract}
\maketitle
\section{Introduction}
Stochastic systems that enjoy some sort of modulation have attracted considerable attention in probability, both from applied and theoretical perspectives. Examples of these are Markov-modulated Poisson processes \cite{fischer1993markov}, Markov additive processes \cite[Chapter XI]{asmussen2003applied} and regime-switching stochastic differential equations \cite[Section II.2]{skorokhod2009asymptotic}, which have found applications in risk theory \cite[Chapter VII]{asmussen2010ruin}, queueing \cite{prabhu1989markov} and finance \cite{elliott2007pricing}, just to name a few. The main appeal behind this framework comes from the flexibility when modelling phenomena whose parameters depend on random environmental factors. For instance, modulation may be used to model correlated catastrophic events in insurance, different workload regimes in a queue, or seasonal sentiment changes in financial markets. In  case the modulation is Markovian, a robust toolbox of matrix-analytic methods has been developed to compute relevant qualities; see e.g. \cite{bladt2017matrix} and references therein. 

In this manuscript, we are interested in a class of processes whose dynamics are described by a continuously varying component which arises as the solution of a stochastic differential equation, and an environmental finite-state component which is modelled after a jump process that may or may not be Markovian. Such a class of processes is referred to in the literature as \emph{hybrid stochastic differential equations} (hybrid SDEs), which generically take the form
\begin{equation}\label{eq:prehybrid}\dd X(t) = \mu( J(t), X(t)) \dd t + \sigma( J(t), X(t))\dd B(t),\end{equation}
where $B$ is a Brownian motion, $J$ is the environmental process, and $\mu$ and $\sigma$ are real functions which satisfy certain regularity conditions.

 In particular, we will be interested in hybrid SDEs which exhibit the following characteristics:
\begin{itemize}

	\item The process $X$ is continuous and one-dimensional.
	\item At a given instant, the environmental process switches state with an intensity that is dependent on the main component.
\end{itemize}

In simpler words, the class of hybrid SDEs we are interested in has components whose evolution is interlaced and dependent on each other. This contrasts with the Markov-modulated case, where one is able to draw a path of the environmental component without any knowledge of the main one (see \cite{nguyen2021wong}). In order to highlight this interlacing feature, in the literature such a model is referred to as hybrid SDEs with state-dependent switching; here we will simply call them hybrid SDEs for brevity.

Existence, uniqueness and stability properties for hybrid SDEs have been extensively studied in recent years \cite{yin2009hybrid,nguyen2016modeling,zhang2020regime}. Another stream of research lies in investigating efficient simulation methods of hybrid SDEs, most of them relying on adapating well-known convergence results into this considerably more challenging scenario \cite{yin2010approximation,nguyen2021wong}. However, computing explicit probabilistic descriptors for such a class of processes has been proven challenging, even in simple scenarios. For instance, several descriptors for Markov-modulated Brownian motion have been explicitly obtained in the literature (see e.g. \cite{asmussen1995stationary,ivanovs2010markov,breuer2012occupation,d2012two,nguyen2022rate}), but virtually none of these results have been extended to hybrid SDEs, even in the Markovian case. Such a task seems intractable simply because the toolbox for general diffusions is considerably more limited than the one available for the Brownian motion.

Our contributions to the literature of hybrid SDEs are the following. In a first instance, we provide a novel pathwise approximation technique by means of a space-grid discretization, resulting in approximations which belong to the class of multi-regime Markov-modulated Brownian motions. The latter are processes that, when restricted to each band of the space-grid, behave like a Markov-modulated Brownian motion. While Wong-Zakai or Euler-Maruyama are the most common approximation methods for simulation of hybrid SDEs in the literature, our proposed approximation sets the course to exploit aspects that are currently only knowm for multi-regime Markov-modulated Brownian motions. As an example, we employ our pathwise approximation result to provide approximations for the descriptors for first exit times of a hybrid SDEs over a band $[0,a]$ , $a>0$, using recent results on the stationary measures of multi-regime Markov-modulated Brownian motion in queueing theory \cite{horvath2017matrix,akar2021transient}. We remark that, to the best of our knowledge, this is the first attempt to compute first passage probabilities and expected occupation times for hybrid SDEs, even when reduced to the case in which the environmental process is Markovian. 

The structure of this paper is as follows. In Section \ref{sec:hybriddef}, we provide the proper framework to construct strong solutions to hybrid SDEs via the uniformization method and concatenation of paths. Later, in Section \ref{sec:approxhSDEs}, we construct the proposed multi-regime Markov-modulated Brownian motion approximation, and prove its uniform convergence in probability to the original hybrid SDE over increasing compact intervals. This approximation result is carried over in two steps by considering an auxiliary process, which serves as a middle point to prove our main result in Theorem \ref{th:mainmain}. In Section \ref{sec:approx} we show how our convergence result can be applied to approximate first passage probabilities of hybrid SDEs, as well as its expected occupation times; some numerical examples of such approximations are explored as well. Finally, in Section \ref{sec:summary} we provide a brief summary of our findings, along with a discussion on some avenues of further research.

\section{Hybrid stochastic differential equations}\label{sec:hybriddef}
Let us provide a precise description of a hybrid SDE and its solution $(J,X)=\{(J(t),X(t))\}_{t\ge 0}$, for which we require a complete probability space $(\Omega, \mathbb{P}, \mathcal{F})$ that supports the following independent components:
\begin{itemize}
  \item A standard Brownian motion ${B}=\{B(t)\}_{t\ge 0}$; this will dictate the continuously-varying nature of $X$.
  \item { A Poisson process ${N}=\{N(t)\}_{t\ge 0}$ of a sufficiently large parameter $\gamma>0$, and a sequence $\{U_\ell\}_{\ell \ge 1}$ of $\mbox{Unif}(0,1)$ i.i.d.  random variables: these will be used to describe the jump dynamics of $J$, with $N$ marking the epochs at which jumps occur, and $U_\ell$ dictating where the $\ell$th jump leads to.}
\end{itemize}
{ Our first goal is to rigorously construct a pair $(J,X)$ that solves the hybrid SDE
\begin{equation}\label{eq:hybrid1}\dd X(t) = \mu( J(t), X(t)) \dd t + \sigma( J(t), X(t))\dd B(t), \quad J(0)=i_0,\; X(0)=x_0,\end{equation}
where $X$ is an a.s.\ continuous real process,
${J}$ is a c\`adl\`ag jump process with finite state space $\mathcal{E}=\{1,\dots,p\}$, $\mu:\mathcal{E}\times \mathds{R} \mapsto \mathds{R}$, $\sigma:\mathcal{E}\times \mathds{R} \mapsto \mathds{R}$, $i_0\in\mathcal{E}$ and $x_0\in\mathds{R}$. 
Furthermore, we require that $(J,X)$ is adapted to the $\mathds{P}$-completed filtration $\mathcal{F}_t$ generated by $\big\{B(s), N(s), U_{N(s)} ; s\le t\big\}$. A pair $(J,X)$ satisfying the aforementioned characteristics is called a \emph{strong solution} of \eqref{eq:hybrid1}. }

\begin{remark}
\normalfont
In the context of classic It\^o diffusions, a strong solution to an SDE needs to be adapted to the filtration generated by the driving noise component $B$ only. In the case of hybrid stochastic differential equations, there is an additional stochastic component, the environmental process $J$, which is  why the term `strong solution' needs to be modified in the hybrid SDE framework.
\end{remark}

In this paper we are interested in hybrid SDEs whose environmental process ${J}$ switches at a state-dependent rate. Specifically, we let $J$ evolve according to
\begin{align}\label{eq:hybrid2}
\mathbb{P}(J(t+h)=j|X(t), J(t)=i)=\delta_{ij} + \Lambda_{ij}(X(t))h + o(h)
\end{align}
for some family of intensity matrices $\{\bm{\Lambda}(x)\}_{x\in \mathds{R}}$ which, at this stage, we assume to be continuous w.r.t.\ $x$.

Note that (\ref{eq:hybrid1}) characterises \emph{part} of the pathwise construction of ${X}$. Indeed if ${J}$ was given, then we could solve ${X}$ during sojourn times of ${J}$ with fixed drift and diffusion coefficients, updating them at each jump time of ${J}$. However, (\ref{eq:hybrid2}) does \emph{not} tell us how ${J}$ is meant to be constructed, it only provides a distributional property, implying that one needs a construction whose characteristics coincide with (\ref{eq:hybrid2}). This would not be a problem if, for instance, $\bm{\Lambda}(x)\equiv\bm{\Lambda}$, $x\in\mathds{R}$, for some intensity matrix $\bm{\Lambda}$; in such a case of a \emph{Markov-modulated diffusion}, the evolution of ${J}$ does not depend on the values of ${X}$, and thus can be constructed via uniformization arguments using the Poisson process ${N}$ and the sequence $\{U_\ell\}_{\ell\ge 1}$ only (see \cite{nguyen2021wong} for more details). In the general case of hybrid SDEs, the construction is more involved since ${J}$ evolves using information provided by ${X}$. Below we discuss a construction developed in \cite{allan2022hybrid} which relies on using a related inhomogeneous uniformization argument. First, we state some assumptions needed.

\begin{assumption}\label{ass:Lipschitz1}
For all $i\in \mathcal{E}$, $\mu(i,\cdot)$ and $\sigma(i,\cdot)$ are Lipschitz-continuous. i.e.\ there exists some $K>0$ such that
\begin{align*}
|\mu(i,x)-\mu(i,y)|\vee|\sigma(i,x)-\sigma(i,y)|\le K|x-y|\quad\mbox{for all}\quad x,y\in\mathds{R}.
\end{align*}
\end{assumption}
\begin{assumption} \label{ass:Bound1}
The family $\{\bm{\Lambda}(x)\}_{x\in\mathds{R}}$ is uniformly bounded, i.e.\  
\[\gamma:=\sup_{x\in\mathds{R}} |\Lambda_{ii}(x)| < \infty.\]
\end{assumption}

Under Assumption \ref{ass:Lipschitz1}, we can guarantee the existence of a unique and strong solution to the (ordinary) stochastic differential equation (SDE)
\begin{equation}\label{eq:strongsol1} 
X^{[i,x,v]}(t) =x + \int_0^t\mu(i,X^{[i,x,v]}(t))\dd r + \int_0^t\sigma(i,X^{[i,x,v]}(t))\dd B^{[v]}(r),\end{equation}
for all $i\in\mathcal{E}$, $x\in\mathbb{R}$ and $v,t\ge 0$, where $B^{[v]}(t):= B(v+t)-B(v)$ is a time-shifted version of ${B}$. In other words, ${X}^{[i,x,v]}$ corresponds to the solution of the SDE driven by ${B}^{[v]}$, with coefficients $\mu(i,\cdot)$ and $\sigma(i,\cdot)$, and starting point $x$. On the other hand, Assumption \ref{ass:Bound1} tells us that the jumps of ${J}$ are \emph{dominated} by those of a Poisson process of intensity $\gamma$, meaning that we can use the arrival points of ${N}$ as the (possible) jump epochs of ${J}$.

The precise construction of $(J,X)$ is as follows. Let $\{\theta_\ell\}_{\ell\ge 0}$ denote the arrival times of ${N}$ (with $\theta_0=0$). Define $X(0)=x_0$ and $J(0)=i_0$, and for $t\in (0,\theta_1)$, let $J(t)=i_0$ and $X(t)=X^{[i_0,x_0, 0]}(t)$: we have defined the processes ${X}$ and ${J}$ in $[0,\theta_1)$. Since we want ${X}$ to be continuous, we ought to define $X(\theta_1)=X(\theta_1-)$. For ${J}$, we decide if it jumps or not at time $\theta_1$ employing the values of $U_1$ and $X(\theta_1)$ as follows:
\begin{equation}\label{eq:Jtheta1}J(\theta_1)= k \quad\mbox{if}\quad U_1\in\left[\sum_{j=1}^{k-1} \delta_{i_0j} + \frac{\Lambda_{i_0 j}(X(\theta_1))}{\gamma},\sum_{j=1}^{k} \delta_{i_0j} + \frac{\Lambda_{i_0 j}(X(\theta_1))}{\gamma} \right).\end{equation}
Although at a first look (\ref{eq:Jtheta1}) may look complex, we are simply using $U_1$ and $X(\theta_1)$ in such a way that ${J}$ jumps to State $k$ at time $\theta_1$ with a mass given by the $(i_0, k)$-th entry of the probability matrix $\bm{I}+\bm{\Lambda}(X(\theta_1))/\gamma$, the uniformized version of $\bm{\Lambda}(X(\theta_1))$ (which we rigorously verify in Lemma \ref{lem:unif0}). After having defined $(J,X)$ in $[0,\theta_1]$, we construct ${X}$ in subsequent intervals $(\theta_\ell,\theta_{\ell + 1}]$ by \emph{concatenating} strong solutions of the type (\ref{eq:strongsol1}) with appropriate chosen values for $(i,x,v)$, as well as using a decision rule similar to (\ref{eq:Jtheta1}) to establish which states ${J}$ visits. More specifically, in a recursive manner, for $\ell=1,2,3,\dots$ and $t\in(0,\theta_{\ell+1}-\theta_\ell)$ let 
\begin{align}
X(\theta_{\ell} + t)& = X^{[i_\ell, x_\ell,\theta_\ell]}(t),\label{eq:constr1}\\
 X(\theta_{\ell+1})& = X(\theta_{\ell+1}-),\label{eq:constr2}\\
J(\theta_{\ell} + t)& = i_\ell,\label{eq:constr3}\\
J(\theta_{\ell+1})& = k\quad\mbox{if}\quad U_\ell\in\left[{C}_{\ell}(k),{C}_{\ell}(k) + {D}_{\ell}(k) \right),\label{eq:constr4}
\end{align}
where $x_\ell=X(\theta_\ell)$, $i_\ell=J(\theta_\ell)$ and
\[{C}_{\ell}(k)=\sum_{j=1}^{k-1}\delta_{i_\ell j} + \frac{{\Lambda}_{i_\ell j}({X}(\theta_{\ell+1}))}{\gamma},\quad {D}_{\ell}(k)=\delta_{i_\ell k} + \frac{{\Lambda}_{i_\ell k}({X}(\theta_{\ell+1}))}{\gamma}.\]
A few aspects that are straightforward to verify about this construction:
\begin{itemize}
	\item ${X}$ and ${J}$ are $\mathcal{F}_t$-adapted,
	\item ${X}$ is continuous at $\theta_1, \theta_2,\dots$, and thus, has a.s.\ continuous paths,
	\item ${J}$ is c\`adl\`ag,
	\item ${X}$ solves the hybrid SDE (\ref{eq:hybrid1}) on every interval $(\theta_\ell,\theta_{\ell+1})$, $\ell=0,1,2,\dots$.
	\end{itemize}
Thus, $(J,X)$ is indeed a strong solution to (\ref{eq:hybrid1}).

\begin{remark}\label{rem:nonunique1}
\normalfont
	Note that ${X}$ is the unique solution to (\ref{eq:hybrid1}) under this \emph{particular} construction of ${J}$. Indeed, if we choose a different construction of ${J}$, the process ${X}$ may take another form. For instance, use the r.v.'s $U_\ell'=1-U_\ell$ instead of $U_\ell$ in (\ref{eq:constr4}); the process ${J}$ will then be different and so will ${X}$.
\end{remark}
We still need to verify that (\ref{eq:hybrid2}) holds, which we prove in a slightly more general scenario next. Recall that (\ref{eq:hybrid2}) only makes sense if $\bm{\Lambda}(x)$ is continuous w.r.t.\ $x$. A generalization of this property for discontinuous $\bm{\Lambda}(\cdot)$ is given by
\begin{align}
&\mathbb{P}\left( J(t+v) = J(t) \mbox{ for all }v\in[0,s]\;\mid\; \mathcal{F}_t\otimes \mathcal{F}^{{B}}_{t+s}, X(t)=x, J(t)=i\right)\nonumber\\
& \quad= \exp\left(\int_0^s\Lambda_{ii}(X^{[i,x,t]}(v))\dd v\right),\label{eq:alternativeJ1}\\
&\mathbb{P}\big(J(\theta_\ell)=j\;\mid\; J(\theta_\ell-)=i, J(\theta_\ell-)\neq J(\theta_\ell), X(\theta_\ell) \big) = \frac{\Lambda_{ij}(X(\theta_\ell))}{|\Lambda_{ii}(X(\theta_\ell))|},\quad i\neq j,\label{eq:alternativeJ2}
\end{align}
where $\mathcal{F}^{{B}}_t$ is the $\mathbb{P}$-completion of the $\sigma$-algebra generated by $\{B(s): 0\le s\le t\}$. In essence, (\ref{eq:alternativeJ1}) and (\ref{eq:alternativeJ2}) correspond to how inhomogeneous Markov jump processes are classically constructed through their integrated jump intensities/hazard rates (see e.g. \cite[Ch. 13]{trivedi2017reliability}). That (\ref{eq:alternativeJ1}) and (\ref{eq:alternativeJ2}) imply (\ref{eq:hybrid2}) for continuous $\bm{\Lambda}(\cdot)$ is readily obtained by noting that 
\[\exp\left(\int_0^h\Lambda_{ii}(X^{[i,x,t]}(v))\dd v\right)=1 + \Lambda_{ii}(x)+o(h),\]
where we employed the continuity of $X^{[i,x,t]}$ too. For sake of completeness, we now show that (\ref{eq:alternativeJ1}) and (\ref{eq:alternativeJ2}) indeed hold; this proof is a simplified version of an analogous result found in \cite{allan2022hybrid}.

\begin{lemma}\label{lem:unif0} Let $(J,X)$ be constructed via (\ref{eq:constr1})-(\ref{eq:constr4}). Then (\ref{eq:alternativeJ1}) and (\ref{eq:alternativeJ2}) hold.
\end{lemma}
\begin{proof}
Integrating with respect to the number and position of arrivals  ${N}$ in $[t,t+s]$, say $\{\theta^{*}_\ell\}_\ell$,
\begin{align}
&\mathbb{P}\left( J(t+v) = J(t) \mbox{ for all }v\in[0,s]\;\mid\; \mathcal{F}_t\otimes \mathcal{F}^{{B}}_{t+s}, X(t)=x, J(t)=i\right)\nonumber\\
&\quad = \sum_{k=0}^\infty \frac{(\gamma s)^k}{k!} e^{-\gamma s}\int_0^s\cdots\int_0^s  \frac{\prod_{j=1}^k\mathds{P}(J(\theta^*_j)=J(\theta^*_j-)\mid \theta^*_j=t+v_j, X^{[i,x,t]}(v_j),  J(\theta^*_j)=i)}{s^k}\dd v_1\dots\dd v_k,\nonumber\\
&\quad = \sum_{k=0}^\infty \frac{(\gamma s)^k}{k!} e^{-\gamma s}\int_0^s\cdots\int_0^s  \frac{\prod_{j=1}^k\left(1+\Lambda_{ii}(X^{[i,x,t]}(v_j))/\gamma\right)}{s^k}\dd v_1\dots\dd v_k\nonumber\\
&\quad = \sum_{k=0}^\infty \frac{(\gamma s)^k}{k!} e^{-\gamma s} \left(1+\frac{\int_0^s\Lambda_{ii}(X^{[i,x,t]}(v))\dd v}{\gamma s}\right)^k=\exp\left(\int_0^s\Lambda_{ii}(X^{[i,x,t]}(v))\dd v\right),\nonumber
\end{align}
where in the first equality we used that on the event $\{J(t+v) = J(t) \mbox{ for all }v\in[0,s], J(t)=i, X(t)=t\}$, $X(t+v)=X^{[i,x,t]}(v)$ for all $v\in[0,s]$ by construction. Finally, for $i\neq j\in\mathcal{E}$
\begin{align*}
& \mathbb{P}\big(J(\theta_\ell)=j\;\mid\; J(\theta_\ell-)=i, J(\theta_\ell-)\neq J(\theta_\ell), X(\theta_\ell) \big)\\
& \quad = \frac{\mathbb{P}\big(J(\theta_\ell)=j\;\mid\; J(\theta_\ell-)=i, X(\theta_\ell) \big)}{\mathbb{P}\big(J(\theta_\ell-)\neq J(\theta_\ell)\;\mid\; J(\theta_\ell-)=i, X(\theta_\ell) \big)} = \frac{\Lambda_{ij}(X(\theta_\ell))}{|\Lambda_{ii}(X(\theta_\ell))|},
\end{align*}
completing the proof.
\end{proof}

We note that the construction of strong solutions for hybrid SDEs can be carried on in more general scenarios. For instance, Brownian-driven multidimensional hybrid SDEs with unbounded jump intensities are considered in \cite{nguyen2016modeling}, while the construction discussed in \cite{allan2022hybrid} concerns multidimensional past-dependent hybrid SDEs driven by L\'evy processes.

\section{Space-grid approximation of hybrid SDEs}\label{sec:approxhSDEs}

Although solutions to state-dependent hybrid SDEs are easy to simulate following the steps in (\ref{eq:constr1})-(\ref{eq:constr4}), studying their distributional properties is a challenging task (see \cite{yin2009hybrid} and references therein). Here we advocate to study hybrid SDEs using an approximation method via discretizing the space-grid.

Let $(\Omega,\mathbb{P},\mathcal{F})$ be as in Section \ref{sec:hybriddef}. For each $i\in\mathcal{E}$, let $\widehat{\mu}(i,\cdot)$ and $\widehat{\sigma}(i,\cdot)$ be a piecewise constant approximation of ${\mu}(i,\cdot)$ and ${\sigma}(i,\cdot)$ on a space-grid $\{\zeta_m\}$ of $\mathds{R}$. In this paper, we assume the following properties.

\begin{assumption}\label{ass:stronghat1}
For each $i\in\mathcal{E}$:
\begin{itemize}
	\item $\widehat{\mu}(i,\cdot)$ and $\widehat{\sigma}(i,\cdot)$ are right-continuous with left limits,
\item $|\widehat{\mu}(i,\cdot)|\le |\mu (i,\cdot)|$ and $|\widehat{\sigma}(i,\cdot)|\le |\sigma(i,\cdot)|$,
\item $\widehat{\mu}$ and $\widehat{\sigma}$ are such that the stochastic differential equation
\begin{equation}\label{eq:SDEhat1}\widehat{X}^{[i,x,v]}(t) =x + \int_0^t\widehat{\mu}(i,\widehat{X}^{[i,x,v]}(t))\dd r + \int_0^t\widehat{\sigma}(i,\widehat{X}^{[i,x,v]}(t))\dd B^{[v]}(r)\end{equation}
admits a strong and unique solution for all $x\in\mathds{R}$ and $v\ge 0$,

\item the space-grid $\{\zeta_m\}$ has a finite cardinality.
\end{itemize}
\end{assumption}
Below we present an elementary stability condition that is useful for our forthcoming developments.
\begin{lemma}\label{lem:finitesecond1}
Let $\widehat{X}^{[i,x,v]}$ and $\widehat{X}^{[i,x,v]}$ be the solutions of \eqref{eq:strongsol1} and \eqref{eq:SDEhat1}, respectively, where $\mu$, $\sigma$, $\wh{\mu}$ and $\wh{\sigma}$ attain Assumptions \ref{ass:Lipschitz1} and \ref{ass:stronghat1}. Then,
\[\mathbb{E}\left(\sup_{s\le t}|X^{[i,x,v]}(s)|^2\right)<\infty\quad\mbox{and}\quad\mathbb{E}\left(\sup_{s\le t}|\wh{X}^{[i,x,v]}(s)|^2\right)<\infty.\]
\end{lemma}
\begin{proof} 
The $L_2$-uniform boundedness over compact intervals of ${X}^{[i,x,v]}$ is a standard result, which we briefly replicate below for the sake of completeness. Let $\rho_m:=\inf\{s>0 \;:\; |{X}^{[i,x,v]}(s)| > m\}$, so that
\[\frac{1}{3}\left|{X}^{[i,x,v]}(s\wedge \rho_m)\right|^2\le |x|^2+\left|\int_0^{s\wedge \rho_m}\mu\left(i, {X}^{[i,x,v]}(r)\right)\dd r\right|^2 + \left|\int_0^{s\wedge \rho_m}\sigma\left(i, {X}^{[i,x,v]}(r)\right)\dd B^{[v]}(r)\right|^2.\]
By Cauchy-Schwartz,
\begin{align*}
\left|\int_0^{s\wedge \rho_m}\mu\left(i, {X}^{[i,x,v]}(r)\right)\dd r\right|^2 \le (s\wedge \rho_m) \int_0^{s\wedge \rho_m}\left|\mu\left(i, {X}^{[i,x,v]}(r)\right)\right|^2\dd r.
\end{align*}
By Doob's inequality and It\^o isometry,
\begin{align*}
\mathbb{E}\left(\sup_{s\le t} \left|\int_0^{s\wedge \rho_m}\sigma\left(i, {X}^{[i,x,v]}(r)\right)\dd B^{[v]}(r)\right|^2\right)\le 2^2 \mathbb{E}\left( \int_0^{s\wedge \rho_m}\left|\sigma\left(i, {X}^{[i,x,v]}(r)\right)\right|^2\dd B^{[v]}(r)\right).
\end{align*}

By the Assumption \ref{ass:Lipschitz1}, $\mu(i,\cdot)$ and $\sigma(i,\cdot)$ must be at most linearly increasing, meaning that there exists some $K'>0$  such that 
\[|\mu(i,x)\vee \sigma(i,x)|\le K'(1+|x|).\]
Then,
\begin{align*}
\mathbb{E}\left(\sup_{s\le t\wedge \rho_m}\left|{X}^{[i,x,v]}(s)\right|^2\right)&\le 3|x|^2 + 3t\int_{0}^t K'\left( \mathbb{E}\left(\sup_{s\le r\wedge \rho_m}\left|{X}^{[i,x,v]}(s)\right|^2\right)+ 1 \right)\dd r \\&\quad + 3\times2^2\int_{0}^t K'\left( \mathbb{E}\left(\sup_{s\le r\wedge \rho_m}\left|{X}^{[i,x,v]}(s)\right|^2\right)+ 1 \right)\dd r\\
&  = K_1' + K_2' \int_0^t\mathbb{E}\left(\sup_{s\le r\wedge \rho_m}\left|{X}^{[i,x,v]}(s)\right|^2\right)\dd r,
\end{align*}
where $K_1'= 3|x|^2 + 3t^2 + 3\times 2^2t$ and $K_2'= K'(3t + 2^2)$. Gronwall's lemma then yields
\begin{align*}
\mathbb{E}\left(\sup_{s\le t\wedge \rho_m}\left|{X}^{[i,x,v]}(s)\right|^2\right)\le K_1'e^{K_2' t};
\end{align*}
by letting $m\rightarrow\infty$, we get that $\mathbb{E}\left(\sup_{s\le t}\left|{X}^{[i,x,v]}(s)\right|^2\right)<\infty$. Now, since $|\widehat{\mu}|\le |\mu|$ and $|\widehat{\sigma}|\le |\sigma|$, then $\widehat{\mu}$ and $\widehat{\sigma}$ are at most linearly increasing as well. Given this, the statement concerning $\widehat{X}^{[i,x,v]}$ follows by analogous arguments to the previously presented.
\end{proof}
\begin{remark}
\normalfont
Finding conditions under which a stochastic differential equation with discontinuous coefficients has a unique strong solution is an active area of research in recent years. For instance, in \cite{krylov2005strong} it is established that \eqref{eq:SDEhat1} has a unique and strong solution for $\widehat{\sigma}(i,\cdot)= 1$ and general measurable function $\widehat{\mu}(i,\cdot)$ that attains certain integrability properties. In \cite{leobacher2017strong}, the same result is established for Lipschitz-continuous functions $\widehat{\sigma}(i,\cdot)$ and piecewise Lipschitz-continuous functions $\widehat{\mu}(i,\cdot)$, both assumed to be bounded. Given the nature of $\widehat{\mu}(i,\cdot)$ and $\widehat{\sigma}(i,\cdot)$ as piecewise constant approximations of $\mu(i,\cdot)$ and $\sigma(i,\cdot)$, and taking into account the existing results in \cite{krylov2005strong} and \cite{leobacher2017strong}, our setup allows for the case in which $\widehat{\sigma}(i,\cdot)$ and $\sigma(i,\cdot)$ are constant, and $\widehat{\mu}(i,\cdot)$ and $\mu(i,\cdot)$ are bounded functions. Other possible setups that produce strong and unique solutions to \eqref{eq:SDEhat1} ought to be studied on a case by case basis, a topic which goes beyond the scope of this manuscript.
\end{remark}

Now, let $\wh{{J}}$ be a c\`adl\`ag $\mathcal{F}_t$-adapted jump process whose jump intensities are of the form 
\begin{equation}\label{eq:levdepJ1}
\mathbb{P}(\wh{J}(t+h)=j|\wh{X}(t), \wh{J}(t)=i)=\delta_{ij} + \wh{\Lambda}_{ij}(\wh{X}(t))h + o(h)
\end{equation}
with $\wh{\bm{\Lambda}}(\cdot)$ being a piecewise constant approximation of $\bm{\Lambda}(\cdot)$ over the space-grid $\{\zeta_m\}$.
Under Assumption \ref{ass:stronghat1} and supposing that the path $\widehat{J}$ is \emph{somehow} known, the solution $\widehat{X}$ to the hybrid SDE
\begin{equation}\label{eq:hybridapp1}\dd \wh{X}(t) = \wh{\mu}( \wh{J}(t), \wh{X}(t)) \dd t + \wh{\sigma}( \wh{J}(t), \wh{X}(t))\dd B(t), \quad \wh{J}(0)=i_0,\; \wh{X}(0)=x_0,\end{equation}
can be constructed by a piecewise concatenation of paths, analogously to (\ref{eq:constr1})-(\ref{eq:constr2}). The resulting process is such that, when restricted to a space interval $(\zeta_m,\zeta_{m+1})$ and to the time intervals for which $\widehat{J}$ is equal to $i$, $\widehat{X}$ behaves like a Brownian motion with drift $\widehat{\mu}(i,\zeta_m)$ and noise coefficient $\widehat{\sigma}(i,\zeta_m)$. In fact, the process $\widehat{X}$ falls within the class of \emph{multi-regime Markov-modulated Brownian motions} studied in \cite{horvath2017matrix} (see \cite{mandjes2003models} for an earlier reference in the case $\widehat{\sigma}\equiv 0$). Heuristically speaking, since the functions $\widehat{\mu}$, $\widehat{\sigma}$ and $\wh{\bm{\Lambda}}$ are space-grid approximations of $\mu$, $\sigma$ and $\bm{\Lambda}$, one can expect that $\wh{X}$ approximates the original solution $X$. A main purpose of this paper is to formalize this statement.

\begin{remark}\label{rem:cons_hatJ} 
\normalfont
In the previous paragraph we did not specify how $\wh{J}$ is constructed. This was on purpose: while it is straightforward to follow analogous steps to those in (\ref{eq:constr3})-(\ref{eq:constr4}) to define $\wh{J}$, here we need to follow slightly different steps in order to guarantee the pathwise convergence of $(\widehat{J},\widehat{X})$ to $(J,X)$. The specifics of such a construction are contained in Subsection \ref{subsec:tildeXhatX}.
\end{remark}

To prove pathwise convergence of $(\wh{J},\wh{X})$ to $({J}, {X})$ as some of its parameters go to $\infty$, we need to proceed with care. First we study the unique continuous strong solution $\tilde{{X}}$ to 
\begin{equation}\label{eq:hybridapptilde1}\dd \tilde{X}(t) = \wh{\mu}( {J}(t), \tilde{X}(t)) \dd t + \wh{\sigma}( {J}(t), \tilde{X}(t))\dd B(t), \quad {J}(0)=i_0,\; \tilde{X}(0)=\tilde{x}_0,\end{equation}
where ${J}$ is the jump process constructed through (\ref{eq:constr1})-(\ref{eq:constr4}), while $X$ is recursively defined by 
\begin{align}
\tilde{X}(\theta_{\ell} + t) = \widehat{X}^{[i_\ell, \tilde{x}_\ell,\theta_\ell]}(t)\quad\mbox{and}\quad\tilde{X}(\theta_{\ell+1})= \widehat{X}(\theta_{\ell+1}-)\quad \mbox{where} \quad\tilde{x}_\ell = \tilde{X}(\theta_{\ell}).\label{eq:constrXtilde0}\end{align} 
The \emph{practicality} of the process $\tilde{{X}}$ is essentially null. $\tilde{{X}}$ is an approximation to ${X}$; however, it needs ${J}$ and ${X}$ to be constructed. This does not mean that it is ill-defined; in fact, once $(J,X)$ is constructed (which can be done without issues following (\ref{eq:constr1})-(\ref{eq:constr4})), then $\tilde{{X}}$ simply arises as a concatenation of solution paths. Instead, the real usefulness of $\tilde{{X}}$ is in providing an intermediate step towards quantifying the approximation error $|{X}-\wh{{X}}|$: we will first measure $|{X}-\tilde{{X}}|$ (discretization error brought by $\widehat{\mu}$ and $\widehat{\sigma}$) and then the discrepancy between $\tilde{{X}}$ and $\wh{{X}}$ (discretization error brought by $\wh{\bm{\Lambda}}$).

\subsection{Measuring $|{X}-\tilde{{X}}|$.}
In order to quantify the approximation error of $\tilde{{X}}$ to $X$, let us first measure the uniform mean-square distance over compact time intervals between ${X}$ and $\tilde{{X}}$ in the case $\mathcal{E}=\{1\}$ (that is, no switching occurs). Abusing notation for a brief moment, let us write $\mu$, $\sigma$, $\widehat{\mu}$ and $\widehat{\sigma}$ instead of $\mu(1,\cdot)$, $\sigma(1,\cdot)$, $\widehat{\mu}(1,\cdot)$ and $\widehat{\sigma}(1,\cdot)$, respectively.

\begin{theorem}\label{th:Eequal1} Suppose that $\mathcal{E}=\{1\}$ and let $t\ge 0$. Then,
\begin{align*}
\mathbb{E}\left(\sup_{s\le t}\left| X(s)-\tilde{X}(s)\right|^2\right)\le C(t) \left(\sup_{x\in \mathbb{R}}\left| \mu(x)-\widehat{\mu}(x)\right|^2+ \sup_{x\in \mathbb{R}}\left| \sigma(x)-\widehat{\sigma}(x)\right|^2+|x_0-\tilde{x_0}|^2\right),
\end{align*}
where $C(t)$ takes the form (\ref{eq:Ctaux1}).
\end{theorem}
\begin{proof}
For $t\ge 0$, define
\[Z(t) = \mathbb{E}\left(\sup_{s\le t}|X(s)-\tilde{X}(s)|^2\right).\]
By Lemma \ref{lem:finitesecond1}, $Z(t)$ is finite for all $t\ge 0$. Using the inequality $(a+b+c)^2\le 3a^2+3b^2+3c^2$,
\begin{align}
Z(t)&\le 3 \mathbb{E}\left(\sup_{s\le t} \left|\int_0^s \mu(X(r)) - \widehat{\mu}(\tilde{X}(r))\dd r\right|^2\right)\\
&\quad + 3\mathbb{E}\left(\sup_{s\le t} \left|\int_0^s \sigma(X(r)) - \widehat{\sigma}(\tilde{X}(r))\dd B(r)\right|^2\right)+ 3|x_0-\tilde{x_0}|^2.\label{eq:Zt1}
\end{align}
Employing the Cauchy-Schwarz inequality we get that
\begin{align}
\mathbb{E}\left(\sup_{s\le t} \left|\int_0^s \mu(X(r)) - \widehat{\mu}(\tilde{X}(r))\dd r\right|^2\right) & \le \mathbb{E}\left(\left(\int_0^t 1\times\left|\mu(X(r)) - \widehat{\mu}(\tilde{X}(r))\right|\dd r\right)^2\right)\nonumber\\
& \le \mathbb{E}\left(\left(\int_0^t 1^2 \dd r\right)\left(\int_0^t \left|\mu(X(r)) - \widehat{\mu}(\tilde{X}(r))\right|^2 \dd r\right)\right)\nonumber\\
& = t \int_0^t \mathbb{E}\left(\left|\mu(X(r)) - \widehat{\mu}(\tilde{X}(r))\right|^2 \right)\dd r,\label{eq:Ztaux1}
\end{align}
while the Burkholder-Davis-Gundy inequality leads to
\begin{align}
\mathbb{E}\left(\sup_{s\le t} \left|\int_0^s \sigma(X(r)) - \widehat{\sigma}(\tilde{X}(r))\dd B(r)\right|^2\right)&\le C_* \mathds{E}\left(\int_0^t \left|\sigma(X(r)) - \widehat{\sigma}(\tilde{X}(r))\right|^2\dd r\right)\nonumber\\
&= C_* \int_0^t \mathds{E}\left(\left|\sigma(X(r)) - \widehat{\sigma}(\tilde{X}(r))\right|^2\right)\dd r\label{eq:Ztaux2}
\end{align}
for some universal constant $C_*>0$. Furthermore, the Lipschitz-continuity of $\mu$ implies
\begin{align}
\left|\mu(X(r)) - \widehat{\mu}(\tilde{X}(r))\right|^2 & \le 2\left|\mu(X(r)) - \mu(\tilde{X}(r))\right|^2 + 2\left|\mu(\tilde{X}(r)) - \widehat{\mu}(\tilde{X}(r))\right|^2\nonumber\\
& \le 2 K^2\left|X(r) - \tilde{X}(r)\right|^2 + 2\sup_{x\in \mathbb{R}}\left| \mu(x)-\widehat{\mu}(x)\right|^2;\label{eq:Ztaux3}
\end{align}
similarly, 
\begin{align}
\left|\sigma(X(r)) - \widehat{\sigma}(\tilde{X}(r))\right|^2
& \le 2 K^2\left|X(r) - \tilde{X}(r)\right|^2 + 2\sup_{x\in \mathbb{R}}\left| \sigma(x)-\widehat{\sigma}(x)\right|^2.\label{eq:Ztaux4}
\end{align}
By virtue of (\ref{eq:Zt1})-(\ref{eq:Ztaux4}), we get
\begin{align*}
Z(t)&\le 6K^2(t+C_*)\int_0^t \mathbb{E}\left(\left|X(r) - \tilde{X}(r)\right|^2\right)\dd r + 6t\sup_{x\in \mathbb{R}}\left| \mu(x)-\widehat{\mu}(x)\right|^2\\
&\quad\quad+ 6t\sup_{x\in \mathbb{R}}\left| \sigma(x)-\widehat{\sigma}(x)\right|^2+3|x_0-\tilde{x_0}|^2\\
&\le 6K^2(t+C_*)\int_0^t Z(r)\dd r + 6t\sup_{x\in \mathbb{R}}\left| \mu(x)-\widehat{\mu}(x)\right|^2+ 6t\sup_{x\in \mathbb{R}}\left| \sigma(x)-\widehat{\sigma}(x)\right|^2+ 3|x_0-\tilde{x_0}|^2.
\end{align*}
Finally, Gronwall's lemma implies that
\begin{align*}
Z(t)&\le C(t)\left(\sup_{x\in \mathbb{R}}\left| \mu(x)-\widehat{\mu}(x)\right|^2+ \sup_{x\in \mathbb{R}}\left| \sigma(x)-\widehat{\sigma}(x)\right|^2 + |x_0-\tilde{x_0}|^2\right)
\end{align*}
where
\begin{equation}\label{eq:Ctaux1}
C(t)=(6t\vee 3)e^{6K^2(t+C_*)t}.
\end{equation}
\end{proof}
As usual, mean-square convergence implies convergence in probability, which in our case takes the following explicit form.
\begin{corollary}\label{cor:convP1}
Suppose that $\mathcal{E}=\{1\}$ and that there exists some $\alpha>0$ such that 
 \begin{equation}\label{eq:assdisc1} \sup_{x\in \mathbb{R}}\left| \mu(x)-\widehat{\mu}(x)\right|\vee\sup_{x\in \mathbb{R}}\left| \sigma(x)-\widehat{\sigma}(x)\right|\vee|x_0-\tilde{x}_0| \le \alpha.\end{equation}
 Then, for any $n\ge 2$,
 \[\mathbb{P}\left(\sup_{s\le t}\left| X(s)-\tilde{X}(s)\right| \ge \sqrt{3C(t)(\log n)}\alpha\right) \le (\log n)^{-1}.\]
\end{corollary}
\begin{proof}
Markov's inequality and Theorem \ref{th:Eequal1} imply that for any $\Delta_{n,t,\alpha}>0$,
\begin{align*}
\mathbb{P}&\left(\sup_{s\le t}\left| X(s)-\tilde{X}(s)\right| \ge\Delta_{n,t,\alpha} \right)  \le \frac{\mathbb{E}\left(\sup_{s\le t}|X(s)-\tilde{X}(s)|^2\right)}{\Delta_{n,t,\alpha}^2}\\
& \le \frac{C(t) \left(\sup_{x\in \mathbb{R}}\left| \mu(x)-\widehat{\mu}(x)\right|^2+ \sup_{x\in \mathbb{R}}\left| \sigma(x)-\widehat{\sigma}(x)\right|^2+|x_0-\tilde{x_0}|^2\right)}{\Delta_{n,t,\alpha}^2} = \frac{3C(t) \alpha^2}{\Delta_{n,t,\alpha}^2}.
\end{align*}
The result follows by choosing $\Delta_{n,t,\alpha}=\sqrt{3C(t)(\log n)}\alpha$.
\end{proof}
Corollary \ref{cor:convP1} implies that we can choose appropriate values of $t$ and $\alpha$ which are dependent on $n$, say $t_n$ and $\alpha_n$, such that 
\begin{align}\label{eq:conditionalphat}
t_n\rightarrow\infty\quad\mbox{and}\quad\sqrt{3 C(t_n)(\log n)}\alpha_n\rightarrow 0,\end{align} and then conclude that the path $\tilde{{X}}$ converges in probability  to ${X}$ uniformly over increasing compact intervals. Here we choose $t_n=\sqrt{\log\log n}$ and $\alpha_n=(\log n)^\beta n^{-\gamma}$ for some $\beta,\gamma > 0$. In such a case, there exists some $n_0\ge 1$ such that for all $n\ge n_0$,
\begin{align}\sqrt{3 C(t_n)(\log n)}\alpha_n&\le (6\sqrt{\log \log n}) e^{12K^2\log\log n} (\log n)^\beta n^{-\gamma}\nonumber\\
& \le (\log n) (\log n)^{12K^2} (\log n)^\beta n^{-\gamma}\nonumber\\
& = (\log n)^{\beta_* + \beta} n^{-\gamma},\label{eq:Ctlognalpha}\end{align}
where $\beta_* = 1+12K^2$; note that $n_0$ does not depend on the choice of $\beta$ or $\gamma$. In light of this, we assume the following.
\begin{assumption}\label{ass:discretize_mu_sigma}There exists some $\beta\ge 0$ and $\gamma>0$ such that for each $n\ge 2$, the approximation $\tilde{{X}}^{(n)}$ (constructed as $\tilde{{X}}$ in \eqref{eq:hybridapptilde1} where the coefficients $\widehat{\mu}$ and $\widehat{\sigma}$ now show an explicit dependence on $n$ via their superscript) satisfies
\[\sup_{x\in \mathbb{R}, i\in\mathcal{E}}\left| \mu(i,x)-\widehat{\mu}^{(n)}(i,x)\right|\vee\sup_{x\in \mathbb{R},  i\in\mathcal{E}}\left| \sigma(i,x)-\widehat{\sigma}^{(n)}(i,x)\right|\vee|x_0-\tilde{x}^{(n)}_0| \le (\log n)^\beta n^{-\gamma}.\]
\end{assumption}

While it might seem that the proof in Theorem \ref{th:Eequal1} can be easily modified to handle the general case $\mathcal{E}=\{1,\dots, p\}$, the switching caused by ${J}$ destroys the local martingale property of the involved stochastic integrals, and thus, the Burkholder-Davis-Gundy inequality no longer applies. For this reason, our approach is to study the discrepancy between ${X}$ and $\tilde{{X}}^{(n)}$ by concatenating paths at the time epochs $\{\theta_\ell\}_{\ell\ge 0}$ for which switching can occur. We explicitly do this in the proof below.
\begin{theorem}\label{th:tildeXmain1}
Let $\mathcal{E}=\{1,\dots, p\}$ and fix some small $\epsilon_1>0$. As $n\rightarrow \infty$, 
\begin{align*}
\mathbb{P}\left(\sup_{s\le \xi_n}\left| X(s)-\tilde{X}^{(n)}(s)\right|> n^{-\gamma+\epsilon_1}\right)=o(1).
\end{align*}
where $\xi_n=\sqrt{\log \log\log n}$.
\end{theorem}
\begin{proof}
First, we claim that 
\begin{equation}\label{eq:Poissontail1}\mathds{P}(N(\xi_n) > \lfloor\log \log n\rfloor)=o(1)\quad\mbox{as}\quad n\rightarrow\infty.\end{equation} Indeed, using \cite[Proposition 1]{glynn1987upper}, for sufficiently large $n$
\begin{align}
\mathds{P}(N(\xi_n) > \lfloor \log\log n\rfloor)\le 
& \left(1-\frac{\gamma\xi_n}{\lfloor\log \log n\rfloor + 1}\right)^{-1}\left(\frac{(\gamma\xi_n)^{\lfloor\log \log n\rfloor}}{\lfloor\log \log n\rfloor!} e^{-\gamma\xi_n}\right):\label{eq:uppertail1}
\end{align}
the first element in the product of the r.h.s. in (\ref{eq:uppertail1}) clearly converges to $1$, while the second element is bounded by
\begin{align*}
&\frac{(\gamma\xi_n)^{\lfloor\log \log n\rfloor}}{\lfloor\log \log n\rfloor!}\\
&\quad \le \left((\gamma\xi_n)^{\lfloor\log\log n\rfloor}\right) \left(\sqrt{2\pi\lfloor\log \log n\rfloor}\left(\frac{\lfloor\log \log n\rfloor}{e}\right)^{\lfloor\log \log n \rfloor} e^{\tfrac{1}{12\lfloor\log\log n\rfloor + 1}}\right)^{-1}\\
&\quad = O \left(\left(\frac{e\gamma \xi_n}{\lfloor\log \log n\rfloor }\right)^{\lfloor\log \log n\rfloor }\right)= o(1).
\end{align*}
Now, fix some large $n\ge 1$ and for all $m\ge 0$ let 
\begin{align*}
A_{m} &=\left\{\sup_{s\le \xi_n\wedge \theta_m}\left| X(s)-\tilde{X}^{(n)}(s)\right| > (\log n)^{m \beta_* + \beta} n^\gamma\right\}.
\end{align*}
Note that
\begin{align*}
\mathbb{P}\left(A_m\right) & = \mathbb{P}(A_{m-1}\cap A_m) + \mathbb{P}(A_{m-1}^c\cap A_m)\\
& \le \mathbb{P}(A_{m-1}) + \mathbb{P}(A_m\mid A_{m-1}^c) = \sum_{\ell=1}^m \mathbb{P}(A_\ell\mid A_{\ell-1}^c)\\
& = \sum_{\ell=1}^m \mathbb{P}\left(\sup_{s\le \xi_n\wedge \theta_\ell}\left| X(s)-\tilde{X}^{(n)}(s)\right| > (\log n)^{\ell \beta_* + \beta} n^\gamma\mid A_{\ell-1}^c \right)\le m(\log n)^{-1},
\end{align*}
where the last inequality follows by using Corollary \ref{cor:convP1} and (\ref{eq:Ctlognalpha}) for all $\widehat{\mu}(i,\cdot)$ and $\widehat{\sigma}(i,\cdot)$, $i\in\mathcal{E}$. Finally, for sufficiently large $n$ (large enough that $(\log \ell)^{\lfloor \log\log \ell\rfloor \beta_* + \beta} \ge n^{\epsilon_1}$ for all $\ell \ge 1$),
\begin{align*}
\left\{\sup_{s\le\xi_n}\left| X(s)-\tilde{X}^{(n)}(s)\right|> n^{-\gamma+\epsilon_1} \right\}\subset \left\{N(\xi_n)>\lfloor \log\log n\rfloor\right\}\cup\bigcup_{\ell=1}^{\lfloor \log\log n\rfloor}A^{(n)}_\ell,
\end{align*}
so that
\begin{align*} 
\mathbb{P}\left(\sup_{s\le\xi_n}\left| X(s)-\tilde{X}^{(n)}(s)\right|> n^{-\gamma+\epsilon_1}\right) \le o(1) + (\lfloor \log\log n\rfloor)^2(\log n)^{-1}=o(1).
\end{align*}
\end{proof}
Now that we have assessed the convergence of $\tilde{X}$ to $X$, we will continue with the discrepancy between $\tilde{X}$ and $\wh{X}$ in the following section, which will ultimately lead to measuring the final error $|{X}-\wh{{X}}|$.

\subsection{Measuring $|\tilde{X}-\widehat{{X}}|$.}\label{subsec:tildeXhatX}
Assessing the discrepancy between $\tilde{X}$ and $\wh{X}$ relies on the following key observation: $\widehat{X}$ is \emph{identical} to $\tilde{X}$ in a compact interval $[0,t]$ if $\widehat{J}$ coincides with $J$ in that interval. Thus, the question of convergence of $\widehat{X}$ to $\tilde{X}$ may be posed as a question of if/when the process $\widehat{J}$ differs from $J$. In order to answer this, we borrow the \emph{decoupling} idea introduced in \cite{allan2022hybrid}. This concept relies on constructing the approximation $\widehat{J}^{(n)}$ (understood as the process $\widehat{J}$ with an explicit dependence on some parameter $n$; we omit such dependence in some places for notational convenience) in such a way that this process decouples or stops being identical to $J$ in an interval $[0,t]$ with a probability that tends to $0$ as $n\rightarrow\infty$. As announced in Remark \ref{rem:cons_hatJ}, instead of constructing $\widehat{J}$ via \eqref{eq:constr3}-\eqref{eq:constr4}, below we describe an alternative (but similar) technique that yields the distributional property \eqref{eq:levdepJ1} as well.


First, let us add an extra process $\{\widehat{H}_\ell\}_{\ell\ge 0}$ which tracks the \emph{decoupling} event of $\widehat{J}$ w.r.t.\ $J$ at the time epochs $\{\theta_\ell\}_{\ell\ge 0}$. Roughly speaking, we say a decoupling occurs at time $\theta_\ell$ if we can no longer guarantee that $\widehat{J}(\theta_\ell)= J(\theta_\ell)$, but we can still guarantee that $\widehat{J}$ is identical to $J$ on $[0,\theta_\ell)$. To properly define this, we consider a process $\{\widehat{H}_\ell\}_{\ell\ge 0}$ taking values in $\{0,1,2\}$, where $\{\widehat{H}_\ell=0\}$ will signify that the decoupling has not occurred by time $\theta_\ell$, $\{\widehat{H}_\ell=1\}$ will signify that the decoupling occured exactly at time $\theta_\ell$, and $\{\widehat{H}_\ell=2\}$ will signify that a decoupling occurred before time $\theta_\ell$. More specifically, let $\widehat{X}=x_0$, $\widehat{J}(\theta_0)=i_0$, $\widehat{H}_0=0$, 
and recursively define for $\ell=0,1,2,\dots$ and $t\in(0,\theta_{\ell+1}-\theta_\ell)$,
\begin{align}
\widehat{X}(\theta_{\ell} + t)& = \widehat{X}^{[\widehat{i}_\ell, \widehat{x}_\ell,\theta_\ell]}(t),\label{eq:constr1a}\\
 \widehat{X}(\theta_{\ell+1})& =\widehat{X}(\theta_{\ell+1}-),\label{eq:constr2a}\\
\widehat{J}(\theta_{\ell} + t)& = \widehat{i}_\ell,\label{eq:constr3a}\\
 \widehat{J}(\theta_{\ell+1})& = k, \; \widehat{H}_{\ell+1}=0 \nonumber\\
&\quad\mbox{if}\quad U_\ell\in \left[{C}_{\ell}(k), {C}_{\ell}(k) + {D}_{\ell}(k)\wedge \widehat{D}_{\ell}(k) \right), \widehat{H}_\ell=0,\label{eq:constr4a0}\\ 
\widehat{J}(\theta_{\ell+1})&\sim (\widehat{D}_{\ell}(k)- {D}_{\ell}(k)\wedge \widehat{D}_{\ell}(k))_{k\in\mathcal{E}}, \; \widehat{H}_{\ell+1}=1 \nonumber\\
&\quad\mbox{if}\quad U_\ell\notin \cup_{k'\in\mathcal{E}}\left[{C}_{\ell}(k'), {C}_{\ell}(k') + {D}_{\ell}(k')\wedge \widehat{D}_{\ell}(k') \right), \widehat{H}_\ell=0,\label{eq:constr4a1}\\
\widehat{J}(\theta_{\ell+1})&\sim (\widehat{D}_{\ell}(k))_{k\in\mathcal{E}}, \; \widehat{H}_{\ell+1}=2\quad\mbox{if}\quad \widehat{H}_\ell\in\{1,2\},\label{eq:constr4a}
\end{align} 

where $\widehat{x}_\ell=\widehat{X}(\theta_\ell)$, $\widehat{i}_\ell=\wh{J}(\theta_\ell)$ and $\widehat{D}_{\ell}(k)=\delta_{i_\ell k} + \frac{\widehat{\Lambda}_{i_\ell k}(\widehat{X}(\theta_{\ell+1}))}{\gamma}$.
Here the notation `$Y\sim\bm{\alpha}$' for $\bm{\alpha}\neq\bm{0}$ denotes `$Y$ is sampled from $\bm{\alpha}/|\bm{\alpha}|$'. 

While Equations \eqref{eq:constr1a}-\eqref{eq:constr3a} are analogous to their counterpart \eqref{eq:constr1}-\eqref{eq:constr3}, note that \eqref{eq:constr4a0}-\eqref{eq:constr4a} are slightly more involved than \eqref{eq:constr4}. Nevertheless, the idea is similar: $\widehat{J}$ is constructed in such a way that at the uniformization epoch $\theta_{\ell+1}$, it lands in State $k$ with probability $\widehat{D}_{\ell}(k)$. Indeed, if $\widehat{H}_\ell= 0$, then in order to land in $k$, either
\begin{itemize}
\item $U_\ell$ is in $\left[{C}_{\ell}(k), {C}_{\ell}(k) + {D}_{\ell}(k)\wedge \widehat{D}_{\ell}(k) \right)$ (with probability ${D}_{\ell}(k)\wedge \widehat{D}_{\ell}(k)$), or
\item $U_\ell$ lands in $[0,1]\setminus\left(\cup_{k'\in\mathcal{E}}\left[{C}_{\ell}(k'), {C}_{\ell}(k') + {D}_{\ell}(k')\wedge \widehat{D}_{\ell}(k')\right)\right)$ and then $k$ gets sampled from the vector $\left(\widehat{D}_{\ell}(k')- {D}_{\ell}(k')\wedge \widehat{D}_{\ell}(k')\right)_{k'\in\mathcal{E}}$ (with probability $\widehat{D}_{\ell}(k)- {D}_{\ell}(k)\wedge \widehat{D}_{\ell}(k)$).\end{itemize} 
In the second case, a decoupling is declared at that step and we let $\widehat{H}_{\ell+1}=1$; under these circumstances, by construction we can still guarantee that $\{\widehat{J}(t)\}_{t < \theta_{\ell+1}} = \{J(t)\}_{t < \theta_{\ell+1}}$ (compare \eqref{eq:constr4} and \eqref{eq:constr4a0}-\eqref{eq:constr4a}) and $\{\widehat{X}(t)\}_{t \le \theta_{\ell+1}} = \{\tilde{X}(t)\}_{t \le \theta_{\ell+1}}$ (compare \eqref{eq:constrXtilde0} and \eqref{eq:constr1a}-\eqref{eq:constr2a}). Once a decoupling occurs, further jumps of $J$ at the uniformization epochs occur by sampling directly from the vector $(\widehat{D}_{\ell}(k'))_{k'\in\mathcal{E}}$, which are points in time at which we can no longer guarantee that $\widehat{J}$ and $J$ are identical (ditto $\widehat{X}$ and $\tilde{X}$). Fortunately, since $\widehat{\Lambda}$ is thought to be close to $\bm{\Lambda}$, so is $\widehat{D}_{\ell}(k)$ to ${D}_{\ell}(k)$, suggesting that the probability of a decoupling at a given step must be low. In the rest of the section, we investigate a stronger statement: decoupling occurs with low probability in any compact time interval.


Similarly to Assumption \ref{ass:discretize_mu_sigma}, let us now make explicit the dependence of $\widehat{\bm{\Lambda}}$, $\widehat{X}$ and $\widehat{J}$ on $n$ via a superscript. In order to obtain convergence rates of the decoupling of $\widehat{J}^{(n)}$ w.r.t.\ $J$, we impose the following regularity condition on $\widehat{\bm{\Lambda}}^{(n)}$.
\begin{assumption}\label{ass:intensities}
For a square matrix $A=\{a_{ij}\}$, let $\Vert\cdot\Vert$ denote the norm defined by
\[\Vert A\Vert=\sup_{i}\left\{\sum_{j}a_{ij}\right\}.\]
The intensity matrix $\bm{\Lambda}(\cdot)$ is $\log$-H\"older continuous, in the sense that there exists a constant $G>0$ such that for any $x,y\in\mathds{R}$,
\[\left\Vert\bm{\Lambda}(x) - \bm{\Lambda}(y)\right\Vert\le \frac{G}{-\log|x-y|}.\]
Furthermore, $\widehat{\bm{\Lambda}}^{(n)}=\widehat{\bm{\Lambda}}$ is right-continuous with left limits and converges uniformly to $\bm{\Lambda}$ at a $\log$-rate, more specifically,
\[\sup_{z\in\mathbb{R}}\left\Vert\widehat{\bm{\Lambda}}^{(n)}(z)- \bm{\Lambda}(z)\right\Vert \le \frac{G}{\log n}.\]
\end{assumption}

\begin{remark}
\normalfont
 Under Assumption \ref{ass:intensities}, $\bm{\Lambda}$ does not admit any discontinuities, however,
 recall that $\log$-H\"older continuity is considerably less restrictive than H\"older- or Lipschitz-continuity. This means that we can \emph{mimic} a discontinuous behaviour for $\bm{\Lambda}$ at a point $z_0$ by considering instead some steep function which behaves like $(\log(z-z_0))^{-1}$ for sufficiently close $z> z_0$.
\end{remark}

\begin{lemma}\label{lem:auxXhat1}
Fix some small $\epsilon_1>0$. Then, 
\begin{equation}\label{eq:aux_result_lemma}\mathbb{P}\left( N(\xi_n)\le \lfloor\log \log n\rfloor, \widehat{H}_{N({\xi_n})}\neq 0, \sup_{s\le \xi_n}\left| X(s)-\widehat{X}^{(n)}(s)\right|\le n^{-\gamma+\epsilon_1}\right)= o(1)\quad\mbox{as}\quad n\rightarrow\infty.\end{equation}
Additionally, Equation \eqref{eq:aux_result_lemma} also holds when $\widehat{X}^{(n)}$ is replaced by $\tilde{X}^{(n)}$.
\end{lemma}
\begin{proof}
For notational convenience, let us write $\widehat{X}$ instead of $\widehat{X}^{(n)}$, and $\widehat{\bm{\Lambda}}$ instead of $\widehat{\bm{\Lambda}}^{(n)}$. Note that
\begin{align}
&\left\{N(\xi_n)\le \lfloor\log \log n\rfloor, \widehat{H}_{N({\xi_n})}\neq 0, \sup_{s\le \xi_n}\left| X(s)-\widehat{X}(s)\right|\le n^{-\gamma+\epsilon_1}\right\}\nonumber\\
&\quad =\bigcup_{\ell=1}^{\lfloor\log \log n\rfloor} \bigcup_{k=1}^\ell\left\{ N(\xi_n) = \ell, \widehat{H}_{k}= 1, \sup_{s\le \xi_n}\left| X(s)-\widehat{X}(s)\right|\le n^{-\gamma+\epsilon_1}\right\}\nonumber\\
&\quad \subseteq \bigcup_{\ell=1}^{\lfloor\log \log n\rfloor} \bigcup_{k=1}^\ell\left\{ N(\xi_n) = \ell,\; \widehat{H}_{k}= 1,\; \sup_{s\le \theta_k}\left| X(s)-\widehat{X}(s)\right|\le n^{-\gamma+\epsilon_1}\right\}.\label{eq:auxunion1}
\end{align}
By conditioning on the history of the processes up to time $\theta_k$ and letting $$F= \{ {N(\xi_n) = \ell,\; \widehat{H}_{k-1}=0,\; \sup_{s\le \theta_k}\left| X(s)-\widehat{X}(s)\right|\le n^{-\gamma+\epsilon_1}}\},$$
we get
\begin{align*}
&\mathbb{P}\left( N(\xi_n)= \ell, \widehat{H}_{k}= 1, \sup_{s\le \theta_k}\left| X(s)-\widehat{X}(s)\right|\le n^{-\gamma+\epsilon_1} \;|\; N(\xi_n), \widehat{H}_{k-1}, \{X(s)\}_{s\le \theta_k}, \{\widehat{X}(s)\}_{s\le \theta_k} \right)\\
&\quad =\mathbb{P}\left(\widehat{H}_{k}= 1 \;|\; N(\xi_n), \{X(s)\}_{s\le \theta_k}, \{\widehat{X}(s)\}_{s\le \theta_k} \right)\times\mathds{1}_{F}\\
&\quad =\mathbb{P}\left(U_{\ell-1}\notin \cup_{k'\in\mathcal{E}}\left[{C}_{\ell-1}(k'), {C}_{\ell-1}(k') + {D}_{\ell-1}(k')\wedge \widehat{D}_{\ell-1}(k') \right) \;|\; N(\xi_n), \{X(s)\}_{s\le \theta_k}, \{\widehat{X}(s)\}_{s\le \theta_k} \right)\times\mathds{1}_{F}\\
&\quad = \left(\sum_{k'\in\mathcal{E}} (\widehat{D}_{\ell-1}(k')-{D}_{\ell-1}(k')\wedge \widehat{D}_{\ell-1}(k'))\right)\times\mathds{1}_{F} \le \left\Vert \frac{{\bm{\Lambda}}({X}(\theta_k))}{\gamma} - \frac{\widehat{\bm{\Lambda}}(\widehat{X}(\theta_k))}{\gamma}\right\Vert\times\mathds{1}_{F}\\
&\quad \le \frac{1}{\gamma}\left[\left\Vert {{\bm{\Lambda}}({X}(\theta_k))} - {{\bm{\Lambda}}(\widehat{X}(\theta_k))}\right\Vert + \left\Vert {{\bm{\Lambda}}(\widehat{X}(\theta_k))} - {\widehat{\bm{\Lambda}}(\widehat{X}(\theta_k))}\right\Vert\right]\times\mathds{1}_{F}\\
&\quad \le \frac{1}{\gamma}\left[\frac{G}{-\log |{X}(\theta_k) - \widehat{X}(\theta_k)|} + \frac{G}{\log n}\right]\times\mathds{1}_{F} \le  \frac{G}{\gamma}\left[\frac{1}{\gamma-\epsilon_1}  + 1\right](\log n)^{-1}.
\end{align*}
From \eqref{eq:auxunion1} we get
\begin{align}
& \mathbb{P}\left( N(\xi_n)\le \lfloor\log \log n\rfloor, \widehat{H}_{N({\xi_n})}\neq 0, \sup_{s\le \xi_n}\left| X(s)-\widehat{X}^{(n)}(s)\right|\le n^{-\gamma+\epsilon_1}\right)\nonumber\\
&\quad \le \sum_{\ell=1}^{\lfloor\log \log n\rfloor} \sum_{k=1}^\ell\mathbb{P}\left( N(\xi_n) = \ell,\; \widehat{H}_{k}= 1,\; \sup_{s\le \theta_k}\left| X(s)-\widehat{X}(s)\right|\le n^{-\gamma+\epsilon_1}\right)\nonumber\\
&\quad \le \sum_{\ell=1}^{\lfloor\log \log n\rfloor} \sum_{k=1}^\ell \frac{G}{\gamma}\left[\frac{1}{\gamma-\epsilon_1}  + 1\right](\log n)^{-1}\le \frac{G}{\gamma}\left[\frac{1}{\gamma-\epsilon_1}  + 1\right]\frac{\lfloor\log \log n\rfloor^2}{\log n}=o(1),\label{eq:bounds_aux3}
\end{align}
which proves \eqref{eq:aux_result_lemma}. Finally, following the same set inclusions in \eqref{eq:auxunion1}, we get
\begin{align*}
& \mathbb{P}\left( N(\xi_n)\le \lfloor\log \log n\rfloor, \widehat{H}_{N({\xi_n})}\neq 0, \sup_{s\le \xi_n}\left| X(s)-\tilde{X}^{(n)}(s)\right|\le n^{-\gamma+\epsilon_1}\right)\\
&\quad \le \sum_{\ell=1}^{\lfloor\log \log n\rfloor} \sum_{k=1}^\ell\mathbb{P}\left( N(\xi_n) = \ell,\; \widehat{H}_{k}= 1,\; \sup_{s\le \theta_k}\left| X(s)-\tilde{X}(s)\right|\le n^{-\gamma+\epsilon_1}\right);
\end{align*}
since on $\{\widehat{H}_k=1\}$ the paths $\{\tilde{X}\}_{s\le\theta_k}$ and $\{\widehat{X}\}_{s\le\theta_k}$ are the same, then 
\begin{align*}
& \mathbb{P}\left( N(\xi_n) = \ell,\; \widehat{H}_{k}= 1,\; \sup_{s\le \theta_k}\left| X(s)-\tilde{X}(s)\right|\le n^{-\gamma+\epsilon_1}\right)\\\quad& = \mathbb{P}\left( N(\xi_n) = \ell,\; \widehat{H}_{k}= 1,\; \sup_{s\le \theta_k}\left| X(s)-\widehat{X}(s)\right|\le n^{-\gamma+\epsilon_1}\right),
\end{align*}
so the result for the case $\tilde{X}$ follows by identical steps to those in \eqref{eq:bounds_aux3}.

\end{proof}
Now we are ready to prove the main result of the paper.
\begin{theorem}\label{th:mainmain}
For any fixed $\epsilon_1>0$,
\begin{align}
\mathbb{P}\left( \left\{\sup_{s\le \xi_n}\left| X(s)-\widehat{X}^{(n)}(s)\right|> n^{-\gamma+\epsilon_1}\right\}\cup\left\{J(s)\neq \widehat{J}^{(n)}(s)\mbox{ for some }s\le \xi_n\right\}\right)=o(1)\quad\mbox{as}\quad n\rightarrow\infty.
\end{align}
\end{theorem}
\begin{proof}
Define the events 
\begin{align*}
&\tilde{F}_1:=\left\{ \sup_{s\le \xi_n}\left| X(s)-\tilde{X}^{(n)}(s)\right|> n^{-\gamma+\epsilon_1} \right\},\quad \widehat{F}_1:=\left\{ \sup_{s\le \xi_n}\left| X(s)-\widehat{X}^{(n)}(s)\right|> n^{-\gamma+\epsilon_1} \right\}\\
 &F_2:=\left\{ \widehat{H}_{N(\xi_n)} \neq 0 \right\},\quad F_3:=\left\{ N(\xi_n)> \lfloor\log \log n\rfloor \right\}.\end{align*}
Given that on the event $F_2^c$ the paths $\{\widehat{X}^{(n)}(s)\}_{s\le\xi_n}$ and $\{\tilde{X}^{(n)}(s)\}_{s\le\xi_n}$ are identical, then
\[\tilde{F}_1\cap F_2^c=\widehat{F}_1\cap F_2^c.\]
Taking this into account, by standard set inclusion/exclusion operations we get
\begin{align}
\widehat{F}_1\cup F_2 = (\widehat{F}_1\cap F_2^c)\cup F_2 = (\tilde{F}_1\cap F_2^c)\cup F_2 = \tilde{F}_1\cup F_2 = \tilde{F}_1\cup (\tilde{F}_1^c\cap F_2).\label{eq:auxsets1}
\end{align}
Employing \eqref{eq:auxsets1} and additional standard set operations lead to the inclusions
\begin{align*}
\widehat{F}_1\cup F_2 & = \{(\widehat{F}_1\cup F_2)\cap F_3^c\}\cup \{(\widehat{F}_1\cup F_2)\cap F_3\}\\
& \subseteq \{(\widehat{F}_1\cup F_2)\cap F_3^c\}\cup F_3\\
& = \{(\tilde{F}_1\cup (\tilde{F}_1^c\cap F_2))\cap F_3^c\}\cup F_3\\
& = (\tilde{F}_1\cap  F_3^c)\cup (\tilde{F}_1^c\cap F_2\cap F_3^c)\cup F_3\\
& \subseteq \tilde{F}_1\cup (\tilde{F}_1^c\cap F_2\cap F_3^c)\cup F_3
\end{align*}
which in turn implies that
\begin{align}
\mathbb{P}&\left( \sup_{s\le \xi_n}\left| X(s)-\widehat{X}^{(n)}(s)\right|> n^{-\gamma+\epsilon_1}\cup \left\{J(s)\neq \widehat{J}^{(n)}(s)\mbox{ for some }s\le \xi_n\right\}\right)\nonumber\\
&\hspace{4.5cm}\le \mathbb{P}(\widehat{F}_1\cup F_2)
\le \mathds{P}(\tilde{F}_1) + \mathds{P}(\tilde{F}_1^c\cap F_2\cap F_3^c) + \mathds{P}(F_3).\label{eq:setaux3}
\end{align}
Since each one of the summands on the r.h.s. of \eqref{eq:setaux3} are shown to be $o(1)$ as $n\rightarrow \infty$ in Theorem \ref{th:tildeXmain1}, Lemma \ref{lem:auxXhat1} and \eqref{eq:setaux3}, respectively, the result follows.
\end{proof}

Theorem \ref{th:mainmain} essentially tells us that under Assumptions \ref{ass:Lipschitz1}-\ref{ass:intensities}, there exists a sequence of multi-regime Markov-modulated Brownian motions $\widehat{X}^{(n)}$ that converge in probability to the solution $X$ of a hybrid stochastic differential equation, as well as their respective underlying processes $\widehat{J}^{(n)}$ and $J$; such a convergence occurs in a uniform sense over increasing compact intervals. In particular, this implies that any first passage times of $X$, say into the set $(-\infty,a)$ or $(a,+\infty)$, can be approximated by the corresponding first passage times of $\widehat{X}^{(n)}$. In the next section, we exploit this to provide efficient approximations to certain first passage probabilities for hybrid stochastic differential equations.

\section{First passage probabilities and expected occupation times for solutions of hybrid SDEs}\label{sec:approx}

Define the first passage times $\tau_0^- = \inf\{t>0: X(t)<0\}$, $\tau_a^+ = \inf\{t>0: X(t)>a\}$ ($a> 0$). For $q\ge 0$, we are interested in providing approximations for the Laplace transform of the first passage times,
\[\psii^-_{ij}(u,q,a)=\mathds{E}(e^{-q(\tau_0^-\wedge\tau_a^+)}\mathds{1}\left\{J(\tau_0^-\wedge\tau_a^+)=j, X(\tau_0^-\wedge\tau_a^+)=0\right\}\,\mid\, J(0)=i, X(0)=u),\]
\[\psii^+_{ij}(u,q,a)=\mathds{E}(e^{-q(\tau_0^-\wedge\tau_a^+)}\mathds{1}\left\{J(\tau_0^-\wedge\tau_a^+)=j, X(\tau_0^-\wedge\tau_a^+)=a\right\}\,\mid\, J(0)=i, X(0)=u),\]
for $0< u < a\le \infty$ and $i,j\in\mathcal{E}$. In essence, the functions $\psii^+_{ij}$ and $\psii^-_{ij}$ characterize the distributional law of the exit times of $X$ from the band $[0,a]$. Employing the law of total probability, it is straightforward to reinterpret $\psii^-_{ij}$ and $\psii^-_{ij}$ as
\[\psii^-_{ij}(u,q,a)=\mathds{P}(\tau_0^-\wedge\tau_a^+<e_q, J(\tau_0^-\wedge\tau_a^+)=j, X(\tau_0^-\wedge\tau_a^+)=0\,\mid\, J(0)=i, X(0)=u),\]
\[\psii^+_{ij}(u,q,a)=\mathds{P}(\tau_0^-\wedge\tau_a^+<e_q, J(\tau_0^-\wedge\tau_a^+)=j, X(\tau_0^-\wedge\tau_a^+)=a\,\mid\, J(0)=i, X(0)=u),\]
where $e_q$ is an $\mbox{Exp}(q)$ random variable independent of everything else; from now on, we adopt this reinterpretation. 

In the existing literature, explicit solutions to $\psii^+_{ij}$ and $\psii^-_{ij}$ only exist in very simple scenarios: either $q=0$ and $\mathcal{E}=\{1\}$ (see e.g. \cite[Chapter 4]{ito2012diffusion}), or $\mu_i$ and $\sigma_i$ are constant for all $i\in\mathcal{E}$ (see e.g. \cite{ivanovs2010markov}). Moreover, we are as well interested in the present value of the expected ocupation times in State $j\in\mathcal{E}$ while below level $b\in (0,a)$, defined as
\[O_{ij}(u,q,a,b)=\mathds{E}\left(\int_{0}^{\tau_0^-\wedge\tau_a^+}e^{-qs}\mathds{1}\left\{J(s)=j,\, 0 < X(s) \le b\right\}\dd s\,\mid\, J(0)=i, X(0)=u\right).\]
Similarly to $\psii^-_{ij}$ and $\psii^+_{ij}$, here we reinterpret $O_{ij}$ as 
\[O_{ij}(u,q,a,b)=\mathds{E}\left(\int_{0}^{\tau_0^-\wedge\tau_a^+\wedge e_q}\mathds{1}\left\{J(s)=j,\, 0 < X(s) \le b\right\}\dd s\,\mid\, J(0)=i, X(0)=u\right).\]

In this section we propose to exploit the approximations laid out in Theorem \ref{th:mainmain}, along with the existing theory for multi-regime Markov-modulated Brownian motions, to provide an efficient approximation scheme based on computing
\[\widehat{\psii}^-_{ij}(u,q,a)=\mathds{P}(\widehat{\tau}^-_0\wedge\widehat{\tau}^+_a<e_q, \widehat{J}(\widehat{\tau}^-_0\wedge\widehat{\tau}_a^+)=j, \widehat{X}(\widehat{\tau}^-_0\wedge\widehat{\tau}_a^+)=0\,\mid\,\widehat{J}(0)=i, \widehat{X}(0)=u),\]
\[\widehat{\psii}^+_{ij}(u,q,a)=\mathds{P}(\widehat{\tau}^-_0\wedge\widehat{\tau}^+_a<e_q, \widehat{J}(\widehat{\tau}^-_0\wedge\widehat{\tau}_a^+)=j, \widehat{X}(\widehat{\tau}^-_0\wedge\widehat{\tau}_a^+)=a\,\mid\,\widehat{J}(0)=i, \widehat{X}(0)=u),\]
\[\widehat{O}_{ij}(u,q,a,b)=\mathds{E}\left(\int_{0}^{\widehat{\tau}_0^-\wedge\widehat{\tau}_a^+\wedge e_q}\mathds{1}\left\{\widehat{J}(s)=j,\, 0<\widehat{X}(s) \le b\right\}\dd s\,\mid\,\widehat{J}(0)=i, \widehat{X}(0)=u\right),\]
where $\widehat{\tau}_0^- = \inf\{t>0:\widehat{X}(t)<0\}$ and $\widehat{\tau}_a^+ = \inf\{t>0:\widehat{X}(t)>a\}$. Indeed, the convergence in probability over increasing compact intervals of Theorem \ref{th:mainmain} yields the convergence in probability of $\widehat{\tau}_0^-$ and $\widehat{\tau}_a^+$ to $\tau_0^-$ and $\tau_a^+$, respectively. Due to the continuity of paths of $\widehat{X}$ and $X$, this readily translates to the pointwise convergence of the functions $\widehat{\psii}^-_{ij}$ and $\widehat{\psii}^+_{ij}$ to ${\psii}^-_{ij}$ and ${\psii}^+_{ij}$. Additionally, under the assumption $q>0$, the dominated convergence theorem implies the convergence of $\widehat{O}_{ij}$ to $O_{ij}$.

While the functions $\widehat{\psii}^-_{ij}$ and $\widehat{\psii}^+_{ij}$ are not readily available in the existing literature, we can compute them by embedding independent copies of the excursion $(\widehat{J}_*, \widehat{X}_*):=\{(\widehat{J}(t),\widehat{X}(t)) : t\le \widehat{\tau}_0^-\wedge \widehat{\tau}_a^+\wedge e_q\}$ into a certain recurrent \emph{queue} on the strip $[0,a]$ (that is, a process reflected on the level boundaries $0$ and $a$). This procedure is similar to the ones exploited in \cite{van2005approximated,yazici2017finite,akar2021transient} to provide finite-time ruin probabilities for multi-regime Markov-modulated risk processes or its subclasses. Below we spell out the details of our construction.

Let $\{(\widehat{J}^{\{\ell\}}_*, \widehat{X}^{\{\ell\}}_*)\}_{\ell\ge 1}$ be a sequence of independent copies of $(\widehat{J}_*, \widehat{X}_*)$. Here we assume that the total length of each excursion $\widehat{X}^{\{\ell\}}_*$, say $T^{\{\ell\}}$, has finite mean; this trivially holds if $q>0$, or if either $\widehat{\tau}_0^-$ or $\widehat{\tau}_a^+$ have a finite mean. Additionally, let our probability space support three independent sequences of $\mbox{Exp}(1)$-distributed i.i.d. random variables $\{\rhoa^{\{\ell\}}\}$, $\{\rhob^{\{\ell\}}\}$ and $\{\rhoc^{\{\ell\}}\}$. Based on the construction of multi-regime queues in \cite{horvath2017matrix,akar2021transient}, here we define $Y$ on a space grid $\{\zeta_m\}_{-M\le m\le M}$ (with $\zeta_{-M}=0$, $\zeta_{0}=u$ and $\zeta_M=a$) modulated by the jump process $L$ that has a state space $\mathcal{E}\cup\{\partial_0\}$, both of which evolve as follows:
\begin{enumerate}
\item On the time interval $[0,\Ta^{\{1\}})$ with $\Ta^{\{1\}}=T^{\{1\}}$, let $(L,Y)$ coincide with $(\widehat{J}^{\{\ell\}}_*,\widehat{X}^{\{1\}}_*)$.
\item At time $\Ta^{\{1\}}$, one of three events happens:
\begin{itemize}
	\item If $\widehat{X}^{\{1\}}(\Ta^{\{1\}})=0$, let $\Tb^{\{1\}} = \Ta^{\{1\}} + \rhoa^{\{1\}}$ and $(Y(t)=0,L(t)=L(\Ta^{\{1\}}-))$ for all $t\in [\Ta^{\{1\}}, \Tb^{\{1\}})$.
	\item If $\widehat{X}^{\{1\}}(\Ta^{\{1\}})=a$, let $\Tb^{\{1\}} = \Ta^{\{1\}} + \rhoc^{\{1\}}$ and $(Y(t)=a,L(t)=L(\Ta^{\{1\}}-))$ for all $t\in [\Ta^{\{1\}}, \Tb^{\{1\}})$.
	\item If $\widehat{X}^{\{1\}}(\Ta^{\{1\}})\in(0,a)$, let $\Tb^{\{1\}} = \Ta^{\{1\}}$.
	\end{itemize}
\item At time $\Tb^{\{1\}}$, one of two events happens:
\begin{itemize}
	\item If $Y(\Tb^{\{1\}})\le u$, let $Y$ increase at unit rate up to reaching level $u$, say, at time $\Tc^{\{1\}}$. Let $L(t)=\partial_0$ for all $t\in[\Tb^{\{1\}}, \Tc^{\{1\}})$.
	\item If $Y(\Tb^{\{1\}})> u$, let $Y$ decrease at unit rate up to reaching level $u$, say, at time $\Tc^{\{1\}}$. Let $L(t)=\partial_0$ for all $t\in[\Tb^{\{1\}}, \Tc^{\{1\}})$.
	\end{itemize}
\item Let $P^{\{1\}}_*=\Tc^{\{1\}} + \rhob{\{1\}}_u$ and $(Y(t)=u,L(t)=\partial_0)$ for all $t\in [\Tc^{\{1\}}, \Td^{\{1\}})$.
\item Repeat Steps (1)-(4) for $\ell=2,3,\dots$ shifting the time accordingly in order to concatenate the excursions along with the increasing sequences $\{\nu_i^{\{\ell\}}\}_{\ell}$, $i\in\{1,2,3,4\}$. This produces a process $(L,Y)$ that regenerates at the epochs $\{\Td^{\{\ell\}}\}_{\ell\ge 1}$.
\end{enumerate}
See Figure \ref{fig:MRMMBM1} for a graphic description of the aforementioned construction. 
\begin{figure}[htbp]
\centerline{\includegraphics[scale=0.45]{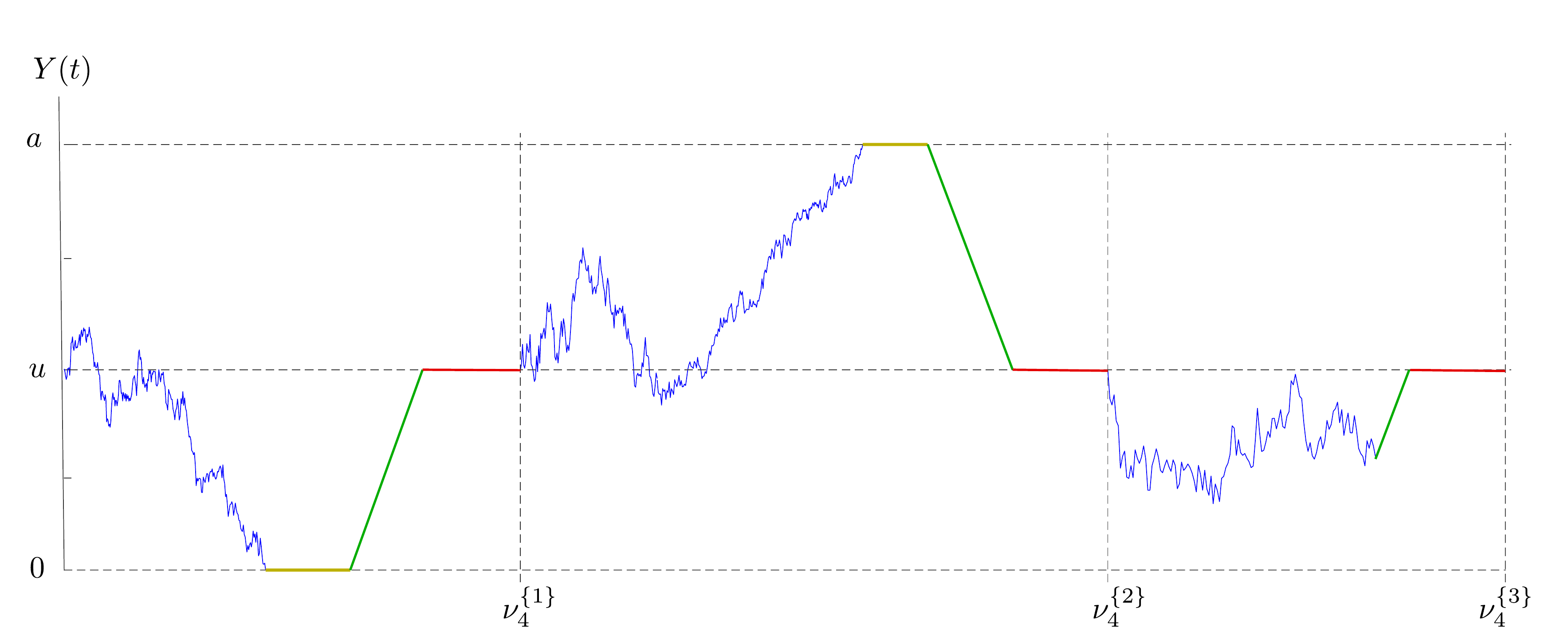}}
\caption{Queue model resulting from concatenating $\{ \widehat{X}^{\{\ell\}}_*\}_{\ell\ge 1}$ using Steps (1)-(4), which are shown in colors blue, mustard, green and red, respectively. }
\label{fig:MRMMBM1}
\end{figure}
As pointed out earlier, the process $(L,Y)$ falls within the class of multi-regime Markov-modulated Brownian motion queueing models (see \cite{horvath2017matrix,akar2021transient} for details), which has a law characterized by the level-dependent matrices $(Q(x), R(x), S(x))_{0\le x\le a}$ with the following interpretation:
\begin{itemize}
	\item $Q_{ij}(x)$ is the jump intensity of $L$ from $i$ to $j$ while $Y$ is in level $x$. We further specify the dependence on the level $x$ by employing intensity matrices $\{Q^{(m)}\}_{-M+1\le m \le  M}$ and $\{\tilde{Q}^{(m)}\}_{-M\le m \le  M}$ where
	\[Q(x)=\left\{\begin{array}{ccc}
	Q^{(m)}& \mbox{for} & \zeta_{m-1} < x <\zeta_m,\\
	\tilde{Q}^{(m)}&\mbox{for}& x= \zeta_m.
	\end{array}\right.\]
	\item $R_{ii}(x)$ is the drift of $Y$ at level $x$ while $L$ is in State $i$ (by convention we let $R_{ij}(x)=0$ for all $i\neq j$). We further specify this dependence on the level $x$ by employing diagonal matrices $\{R^{(m)}\}_{-M+1\le m \le  M}$ and $\{\tilde{R}^{(m)}\}_{-M\le m \le  M}$ where
	\[R(x)=\left\{\begin{array}{ccc}
	R^{(m)}& \mbox{for} & \zeta_{m-1} < x <\zeta_m,\\
	\tilde{R}^{(m)}&\mbox{for}& x= \zeta_m.
	\end{array}\right.\]
	\item $S_{ii}(x)$ is the diffusion coefficient of $Y$ at level $x$ while $L$ is in State $i$ (by convention we let $S_{ij}(x)=0$ for all $i\neq j$). We further specify the dependence on the level $x$ by employing nonnegative diagonal matrices $\{S^{(m)}\}_{-M+1\le m \le  M}$ where
	\[S(x)=S^{(m)}\quad\mbox{for}\quad\zeta_{m-1} \le x <\zeta_m.\]
	\end{itemize}

	Under the previous considerations, the corresponding matrices for our particular model with state space $\mathcal{E}\cup\{\partial_0\}$ take the form 
	\begin{align*}
		Q^{(m)}&=\begin{pmatrix}\widehat{\Lambda}(\zeta_{m-1})-q\bm{I}& q\bm{1} \\ \bm{0} & 0\end{pmatrix},\quad m\in \{-M+1, -M+2,\dots, M\},\\
	\tilde{Q}^{(m)}&=\begin{pmatrix}\widehat{\Lambda}(\zeta_{m})-q\bm{I}& q\bm{1}\\ \bm{0}& 0 \end{pmatrix},\quad m\in \{-M+1, -M+2,\dots, M-1\}\setminus\{0\},\\
	\tilde{Q}^{(-M)}&=\tilde{Q}^{(M)}=\begin{pmatrix}-\bm{I}& \bm{1} \\ \bm{0} & 0\end{pmatrix},\quad \tilde{Q}^{(0)}=\begin{pmatrix}\widehat{\Lambda}(\zeta_{0})-q\bm{I}& q\bm{1}\\ \bm{e}^{\intercal}_{i}& -1\end{pmatrix},
	\end{align*}
		\begin{align*}
		R^{(m)}&=\begin{pmatrix}\mbox{diag}\{\widehat{\mu}_i(\zeta_{m-1})\}& \bm{0} \\ \bm{0} & 1\end{pmatrix},\quad m\in \{-M+1, -M+2,\dots, 0\},\\
		R^{(m)}&=\begin{pmatrix}\mbox{diag}\{\widehat{\mu}_i(\zeta_{m-1})\}& \bm{0} \\ \bm{0} & -1\end{pmatrix},\quad m\in \{1, 2,\dots, M\},\\
	\tilde{R}^{(m)}&=\begin{pmatrix}\mbox{diag}\{\widehat{\mu}_i(\zeta_{m})\} & \bm{0}\\ \bm{0}& 1 \end{pmatrix},\quad m\in \{-M+1, -M+2,\dots, -1\},\\
	\tilde{R}^{(m)}&=\begin{pmatrix}\mbox{diag}\{\widehat{\mu}_i(\zeta_{m})\} & \bm{0}\\ \bm{0}& -1 \end{pmatrix},\quad m\in \{1, -M+2,\dots, M-1\},\\
	\tilde{R}^{(-M)}&=\begin{pmatrix}\bm{0} & \bm{0}\\ \bm{0}& 1 \end{pmatrix},\quad \tilde{R}^{(M)}=\begin{pmatrix}\bm{0} & \bm{0}\\ \bm{0}& -1 \end{pmatrix},\quad \tilde{R}^{(0)}=\begin{pmatrix}\mbox{diag}\{\widehat{\mu}_i(\zeta_{0})\} & \bm{0}\\ \bm{0}& 0 \end{pmatrix},
	\end{align*}
	\begin{align*}
	S^{(m)}&=\begin{pmatrix}\mbox{diag}\{\widehat{\sigma}_i(\zeta_{m-1})\}& \bm{0} \\ \bm{0} & 0\end{pmatrix},\quad m\in \{-M+1, -M+2,\dots, M\},
	\end{align*}
	where $\bm{e}^{\intercal}_{i}$ denotes the $i$th canonical row vector, and $\mbox{diag}\{a_i\}$ denotes the diagonal matrix with the elements $\{a_1,\dots,a_p\}$ filling the diagonal. In \cite{horvath2017matrix}, the authors provide an efficient algorithm to compute the steady state distribution of $(L,Y)$. In particular, using their method we are able obtain:
  \begin{itemize}
    \item the steady state probability atoms for levels $\zeta_{-M}$ and $\zeta_{M}$ while on the states $\mathcal{E}$ in row vector form, say $\bm{p}_{-M}$ and $\bm{p}_{M}$ where
    \begin{align*}
    (\bm{p}_{-M})_j&=\lim_{t\rightarrow \infty}\mathds{P}(L(t)=j, Y(t)= 0),\\
     (\bm{p}_{M})_j&=\lim_{t\rightarrow \infty}\mathds{P}(L(t)=j, Y(t)= a).
    \end{align*}
    \item the steady state probability atoms for level $\zeta_0$ while on the State $\partial_0$, say $p_0$ where
    \begin{align*}
    p_0&=\lim_{t\rightarrow \infty}\mathds{P}(L(t)=\partial_0, Y(t)= u).
    \end{align*}
    \item the steady state probability distribution function while on State $j\in\mathcal{E}$ over $(0,a)$, say $F_j$ where
    \[F_j(b)=\lim_{t\rightarrow \infty}\mathds{P}(L(t)=j, 0<Y(t)\le b),\quad b\in (0,a).\]
  \end{itemize} Below we link these quantities with the first passage probabilities $\widehat{\psii}^+_{ij}$ and $\widehat{\psii}^-_{ij}$, as well as with the expected occupation times $\widehat{O}_{ij}$.

\begin{theorem}\label{th:firstpas1}
Suppose that $\mathds{E}(\widehat{\tau}_0^-\wedge \widehat{\tau}_a^+\wedge e_q\,\mid\, \widehat{J}(0)=i, \widehat{X}(0)=u)<\infty$. Then,
\begin{equation}\label{eq:formpsis}\widehat{\psii}^-_{ij}(u,q,a)=\frac{(\bm{p}_{-M})_j}{p_0}\quad\mbox{and}\quad\widehat{\psii}^+_{ij}(u,q,a)=\frac{(\bm{p}_{M})_j}{p_0}.\end{equation}
 Moreover,
\begin{equation}\label{eq:formO}\widehat{O}_{ij}(u,q,a,b)=\frac{F_j(b)}{p_0},\quad b\in(0,a).\end{equation}
\end{theorem}
\begin{proof}
The process $(L,Y)$ regenerates at the epochs $\{\Td^{\{\ell\}}\}_{\ell\ge 1}$, all of which have a (common) finite first moment. This in turn implies that $(L,Y)$ is positive recurrent, so by \cite[Theorem 1]{glynn1994some},
\begin{align*}
\bm{p}_M\bm{1} &= \frac{\mathds{E}\left(\int_0^{\Td^{\{1\}}}\mathds{1}\{Y(s)=u, L(s)\in \mathcal{E}\}\dd s\right)}{\mathds{E}\left(\Td^{\{1\}}\right)}\\
& = \frac{\mathds{P}(Y(\Ta^{\{1\}})=u, L(\Ta^{\{1\}})\in \mathcal{E})\mathds{E}\left(\rhob^{\{1\}}\right)}{\mathds{E}\left(\Td^{\{1\}}\right)} =\frac{\widehat{\psii}^+(u,q,a)}{\mathds{E}\left(\Td^{\{1\}}\right)},
\end{align*}
where we used that $\mathds{E}\left(\rhob^{\{1\}}\right)=1$.
Employing similar arguments we get
\begin{align*}
\bm{p}_{-M}\bm{1} & =\frac{\widehat{\psii}^-(u,q,a)}{\mathds{E}\left(\Td^{\{1\}}\right)}\quad\mbox{and}\quad p_0=\frac{1}{\mathds{E}\left(\Td^{\{1\}}\right)},
\end{align*}
so that \eqref{eq:formpsis} follows. Additionally,
\begin{align*}
F_j(b) & = \frac{\mathds{E}\left(\int_0^{\Td^{\{1\}}}\mathds{1}\{ L(s)=j,\, 0<Y(s)\le b \}\dd s\right)}{\mathds{E}\left(\Td^{\{1\}}\right)}\\
&= \frac{\mathds{E}\left(\int_{0}^{\widehat{\tau}_0^-\wedge\widehat{\tau}_a^+\wedge e_q}\mathds{1}\left\{\widehat{J}(s)=j,\, 0<\widehat{X}(s) \le b\right\}\dd s\right)}{\mathds{E}\left(\Td^{\{1\}}\right)} = \frac{\widehat{O}_{j,b}(u,q,a)}{\mathds{E}\left(\Td^{\{1\}}\right)},
\end{align*}
which in turn implies \eqref{eq:formO}.
\end{proof}
In short, with the help of the algorithms developed in \cite{horvath2017matrix}, one is able to efficiently compute first passage probabilities and their expected occupation times for $(\widehat{J},\widehat{X})$ using Theorem \ref{th:firstpas1}, which in turn approximates the first passage probabilities and their expected occupation times for $(J,X)$ by virtue of Theorem \ref{th:mainmain}. Below we explore a couple of synthetic examples.

\begin{example}\label{ex:hybrid1} 
\normalfont
Consider a three-state hybrid SDE $(X,J)$ evolving in the space-band $(0,1)$ with parameters
\begin{align*}
&\mu_1(x)= 0.5,\quad\mu_2(x)=0.5(1-x),\quad \mu_3(x)=0.5(1-x)^2,\\
&\sigma_1=\sigma_2=\sigma_3=1.
\end{align*}
In essence, the process $X$ has a force pushing it upwards while $J$ is in any of the three states. The force is weakest while in State $3$ and strongest while in $1$; note that the difference between regimes is accentuated as $X$ gets closer to level $1$. In the context of actuarial science, this may e.g. describe the case of the surplus process of an insurance company with three possibly occurring regimes, all with positive drift, but in some being more responsive with premium reductions when the surplus approaches $1$. 
This can be described, for instance, by an intensity matrix function of the form
\[\bm{\Lambda}(x)=10\times\begin{pmatrix}- x & x & 0 \\ (1-x) & -1  & x\\ 0 &  1-x & -(1-x)\end{pmatrix}.\]
Employing the Matlab coded provided in \url{http://www.hit.bme.hu/~ghorvath} based on \cite{horvath2017matrix} together with Theorem \ref{th:firstpas1}, we are able to compute the first passage probabilities and expected occupation times of the multi-regime Markov-modulated Brownian motion $(\widehat{J},\widehat{X})$ that approximates $(J,X)$. We perform such computations for $q=0$, $a=1$, $i=2$ and $M=50$, with results shown in Figure \ref{fig:Ex1UandOccup}.
\begin{figure}[h!]
  \centering
  \begin{subfigure}[b]{0.45\textwidth}
        \includegraphics[width=\textwidth]{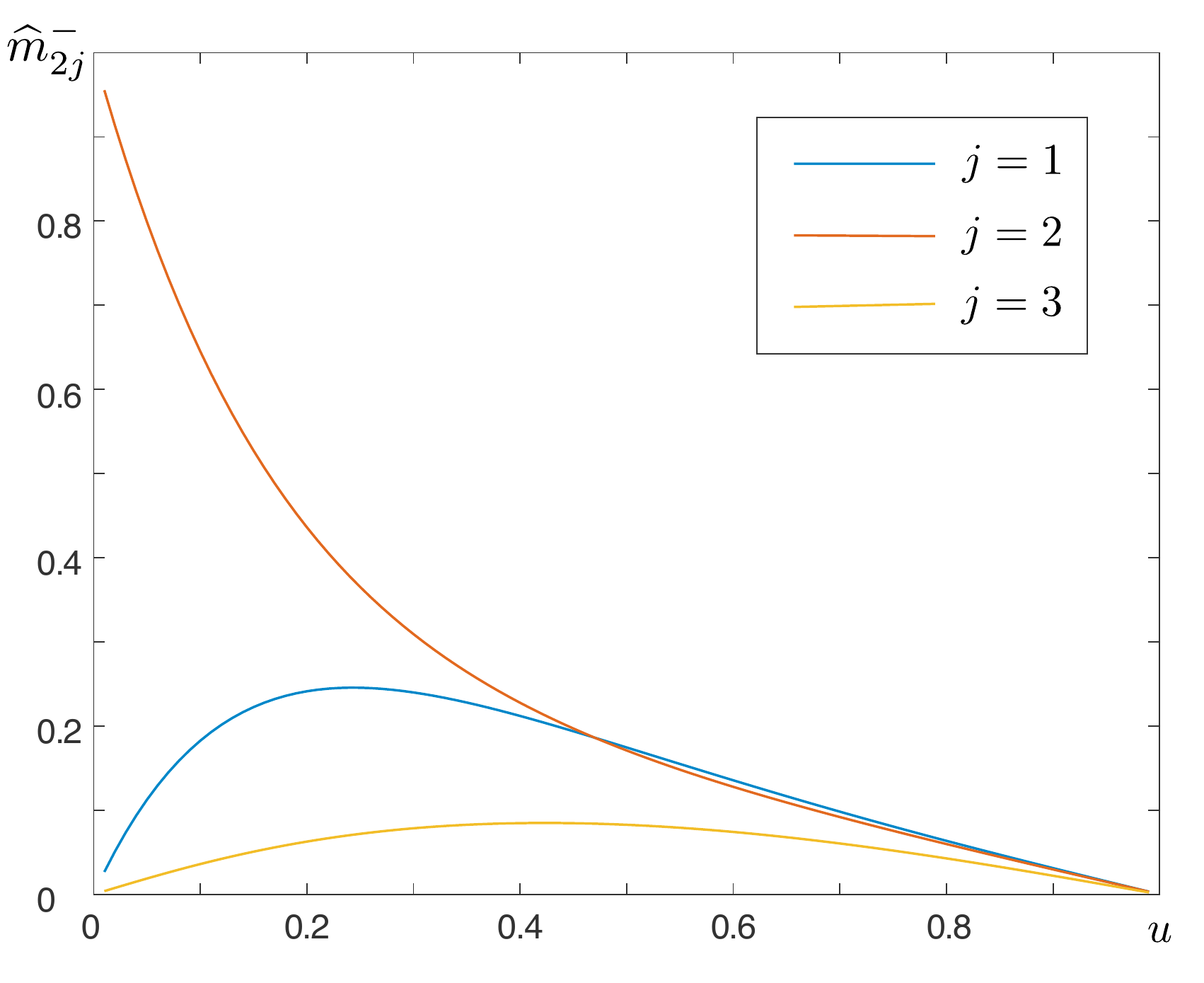}
    \end{subfigure}\qquad
    \begin{subfigure}[b]{0.45\textwidth}
        \includegraphics[width=\textwidth]{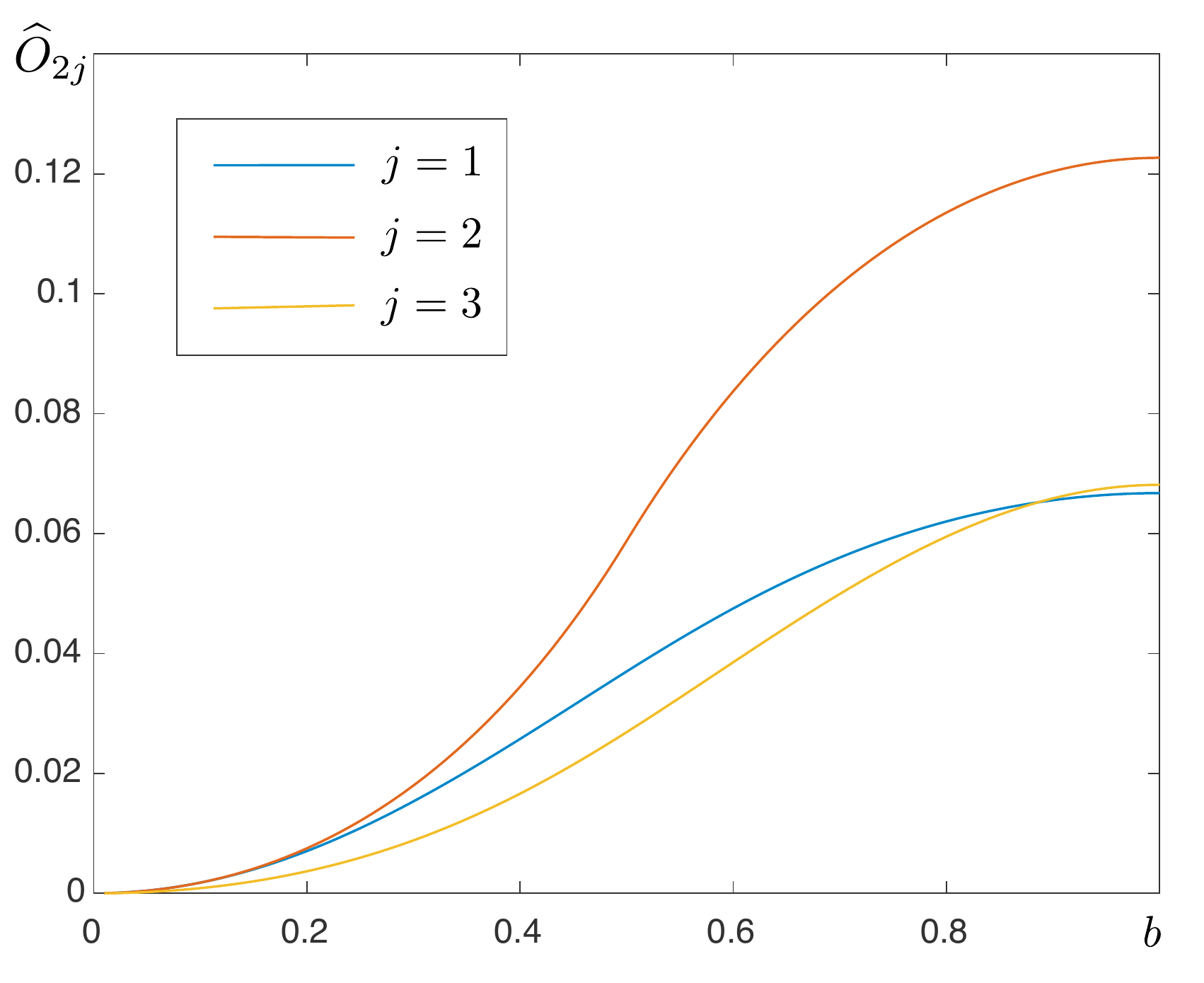}
    \end{subfigure}
    \caption{Left: Plots of $\widehat{\psii}^-_{2j}$ as a function of $u$. Right: Plots of $\widehat{O}_{2j}$ as a function of $b$ with $u=0.5$. Each plot contains the cases $j=1,2,3$.}
    \label{fig:Ex1UandOccup}
\end{figure}
For small values of $u$, downcrossings of level $0$ are most likely while in State $2$; this can be explained by noting that the random oscillations while in State $2$ (the initial state of $J$) can produce a downcrossing before a switching even occurs. For medium values of $u$, the probabilities of downcrossing level $0$ while in States $1$ and $2$ are comparable, both of which are considerably larger than that corresponding to State $3$; this is due to the state-dependent switching behaviour of $\Lambda$ which favours States $1$ and $2$ while the level is low, and $2$ and $3$ while the level is high. For high values of $u$, the probability of downcrossing $0$ is uniformly low for all states. On the other hand, we can observe that the expected occupation time while in State $2$ is larger than that in State $1$ or $2$ uniformly over all $u$; this is a consequence of State $2$ being the initial state of the system.

While Theorem $\ref{th:mainmain}$ guarantees that $\widehat{\psii}^-_{ij}$ converges to $\psii^-_{ij}$ as the space-grid becomes denser, in Figure \ref{fig:VaryingM} we (empirically) show how fast this convergence happens.
\begin{figure}[h!]
  \centering
        \includegraphics[width=0.5\textwidth]{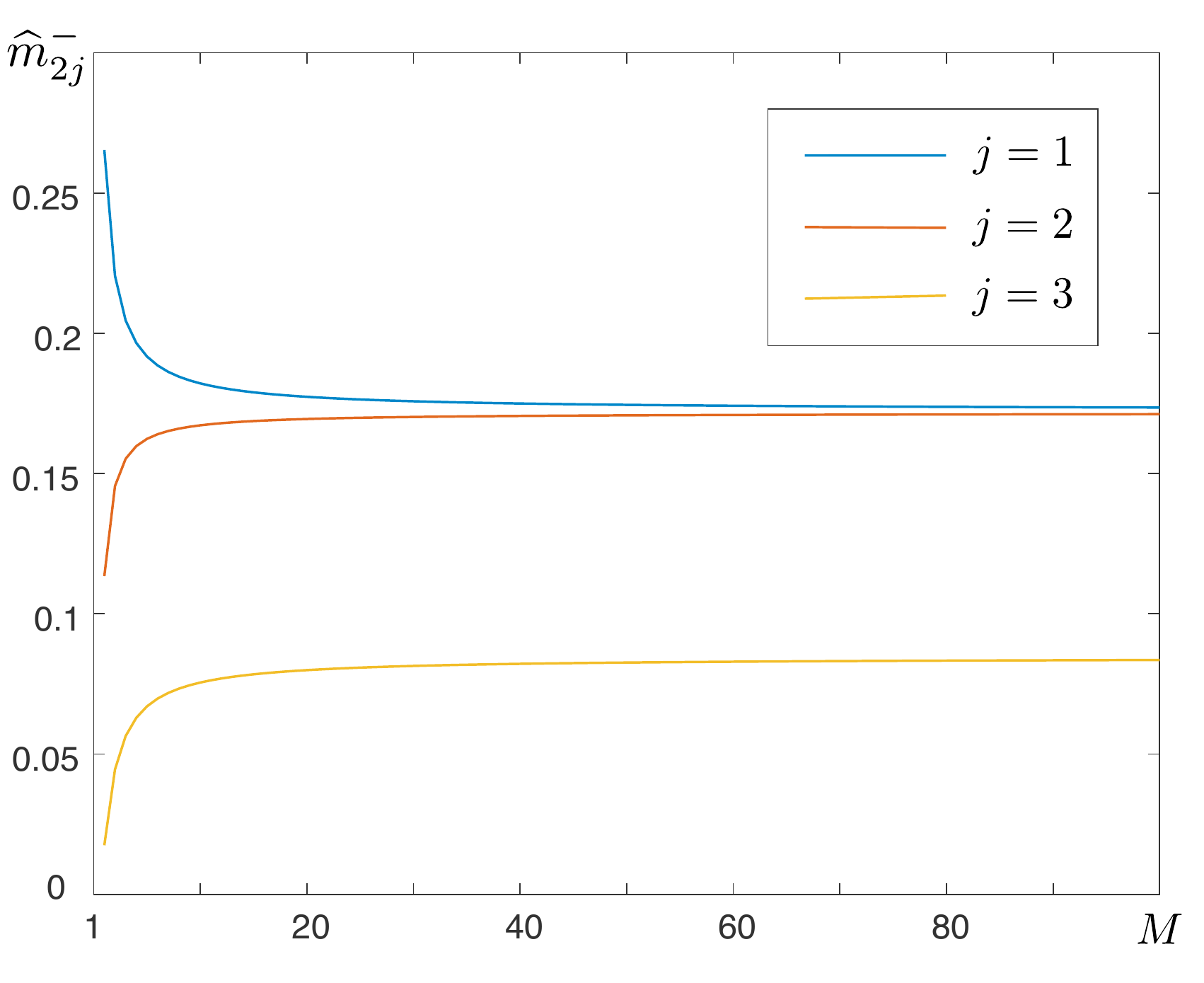}
    \caption{Plots of $\widehat{\psii}^-_{2j}$ as a function of $M$ for the cases $j=1,2,3$.}
    \label{fig:VaryingM}
\end{figure}
As we can appreciate from the plot, convergence is quickly achieved, with differences between the cases being essentially negligible for $M\ge 40$.
\end{example}

\begin{example} \label{ex:hybrid2}
\normalfont
Consider now the same parameters as in Example \ref{ex:hybrid1}, but take $\mu_3(x)=-0.5x^2$ and $\sigma_3=0$. While this alternative scenario does not necessarily have an immediate practical interpretation, it does help to investigate some of the consequences of having a state which lacks any random noise behaviour. From Figure \ref{fig:Ex2UandOccup}, we note that the downcrossing probabilities for States $1$ and $3$ are larger than those in Example \ref{ex:hybrid1}; this is because now we have one of the states pushing the level process downwards. However, note that the probability of downcrossing $0$ while in State $3$ is null; this is because $\mu_3$ approaches $0$ as the level gets closer to $0$, meaning that the process simply cannot cross level $0$ while in State $3$. Moreover, we note that for large values of $b$, the expected occupation time while in State $3$ is larger than in States $1$ or $2$; this suggests that before exiting the band $(0,1)$, the process switches to and stays in State $3$ for a larger amount of time than in Example \ref{ex:hybrid1}.
\begin{figure}[h!]
  \centering
  \begin{subfigure}[b]{0.45\textwidth}
        \includegraphics[width=\textwidth]{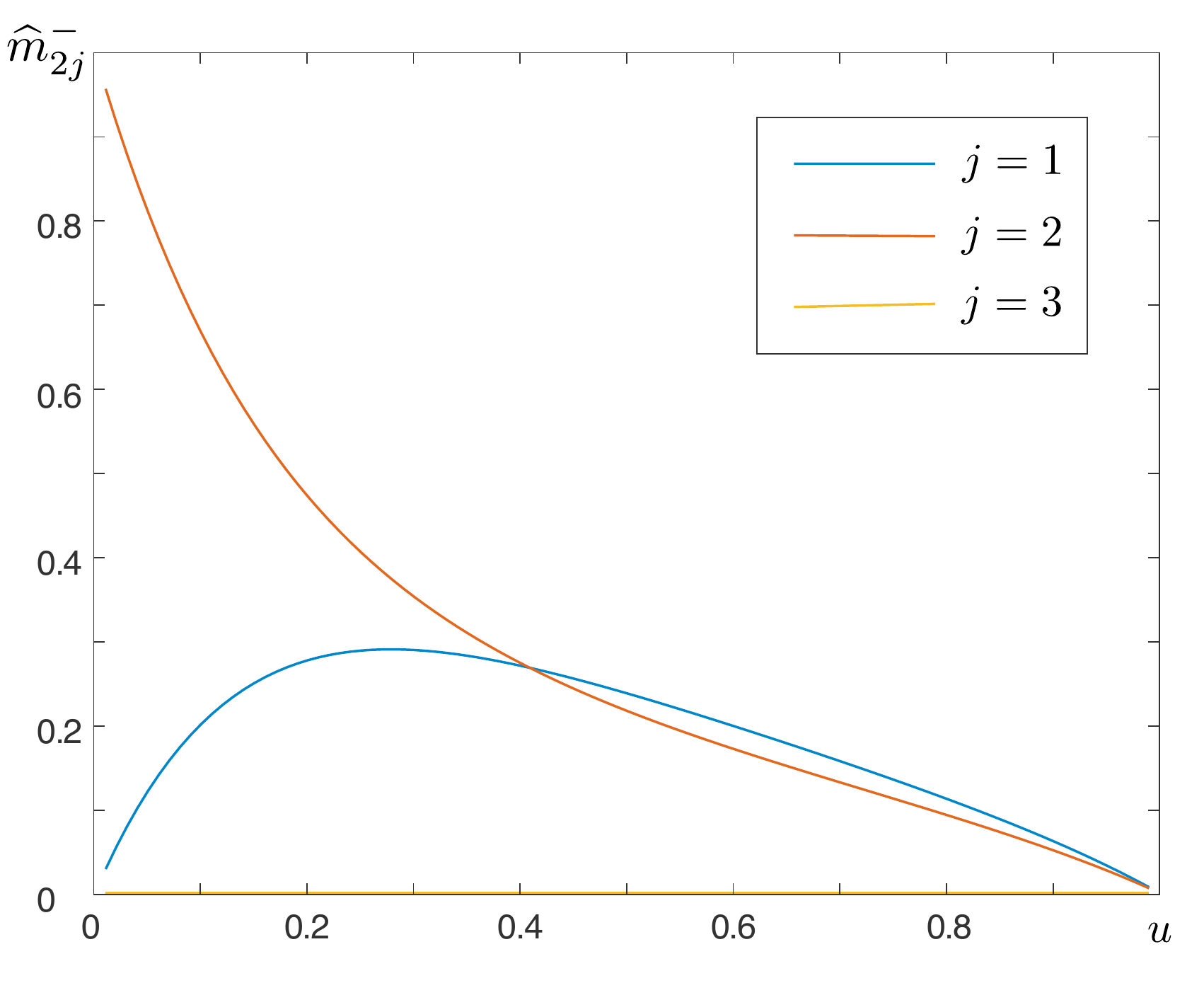}
    \end{subfigure}\qquad
    \begin{subfigure}[b]{0.45\textwidth}
        \includegraphics[width=\textwidth]{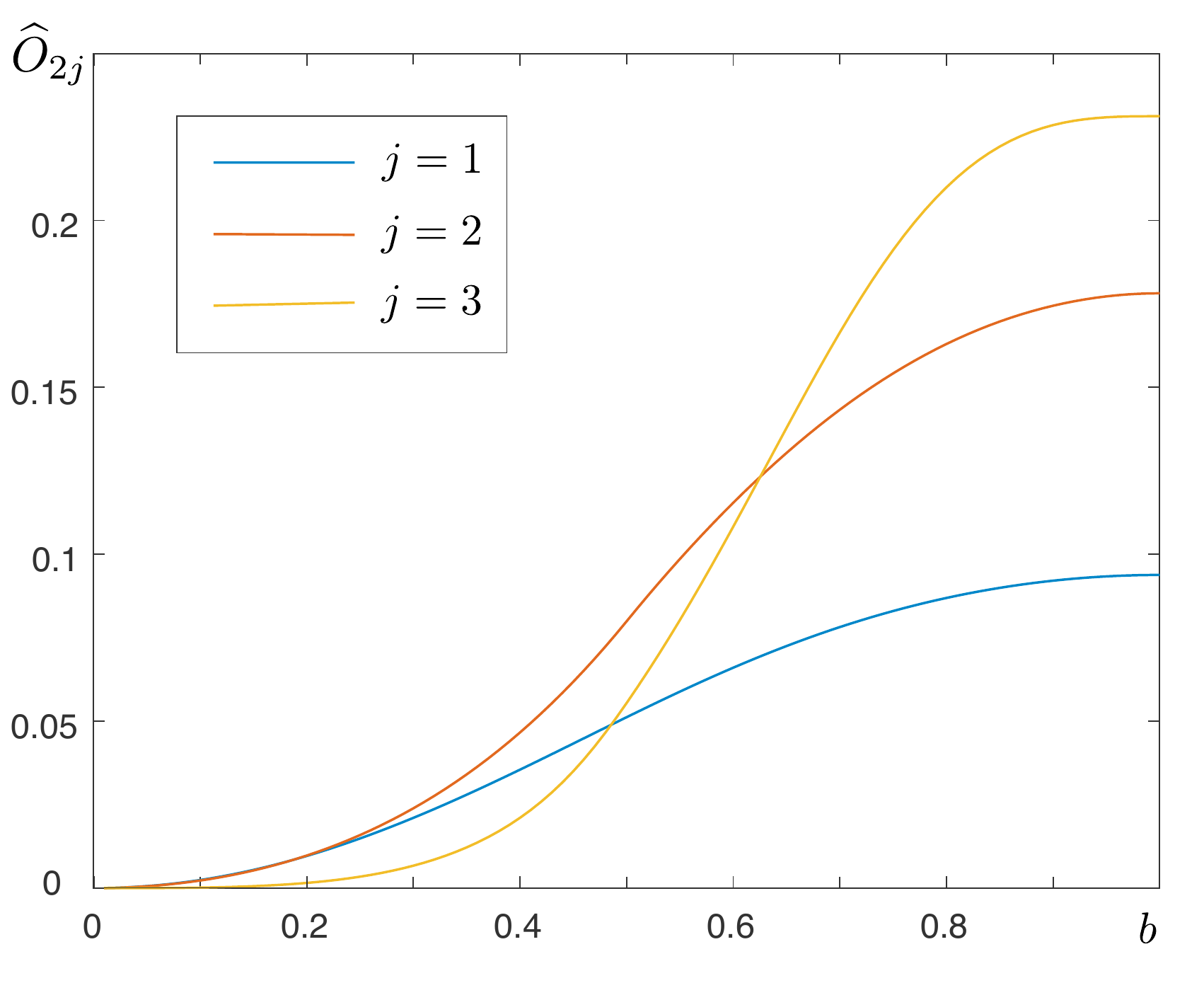}
    \end{subfigure}
    \caption{Left: Plots of $\widehat{\psii}^-_{2j}$ as a function of $u$. Right: Plots of $\widehat{O}_{2j}$ as a function of $b$ with $u=0.5$. Each plot contains the cases $j=1,2,3$.}
    \label{fig:Ex2UandOccup}
\end{figure}
\end{example}

\section{Summary and extensions}\label{sec:summary}
Under Assumptions \ref{ass:Lipschitz1}-\ref{ass:intensities}, we considered the class of hybrid SDEs and their space-grid approximations, the latter belonging to a the family of multi-regime Markov-modulated Brownian motions. We provided a rigorous proof of their pathwise convergence, which holds in a probability sense uniformly over increasing compact intervals. As an application of our convergence result, we considered the challenging problem of identifying first passage probabilities and expected occupation times for solutions of hybrid SDEs, for which we suggest an approximate and computationally efficient answer which employs developments in the area of multi-regime Markov-modulated queues. 

We point out that applications of our convergence result may also be used to approximate other (and possibly more complex) descriptors of hybrid SDEs. For instance, one can easily build hybrid SDEs with phase-type jumps by means of the fluidization method \cite{badescu2005risk}. Other modifications and extensions which are straightforward to implement in our framework are the Omega model \cite{albrecher2011optimal,gerber2012omega}, the Erlangian approximations for finite-time probabilities of ruin \cite{asmussen2002erlangian}, and Parisian ruin problems with phase-type clocks \cite{bladt2019parisian}.

\textbf{Acknowledgement.} The authors would like to acknowledge financial support from the Swiss National Science Foundation Project 200021\_191984.


\bibliographystyle{abbrv}
\bibliography{firstHybrid.bib}

\end{document}